\RequirePackage{fix-cm}

\documentclass{svjour3}

\smartqed

\usepackage{amsmath,amsfonts} 
\usepackage{graphicx}
\usepackage{color}
\usepackage{subfigure}

\usepackage {tikz}
\usepackage{pgfkeys}
\usepackage[hmargin=2.5cm,vmargin=3cm,bindingoffset=0.5cm]{geometry}
\DeclareMathOperator{\dom}{dom}

\DeclareMathOperator{\rank}{rank}
\DeclareMathOperator{\trace}{Tr}
\newcommand{\Tr}[1]{{\trace\left(#1\right)}}
\newcommand{\norm}[1]{{\left\lVert#1\right\rVert}}

\newcommand{\figref}[1]{\figurename~\ref{#1}}

\journalname{Mathematical Programming}
\begin{document}

\title{Smooth Strongly Convex Interpolation and Exact Worst-case Performance of First-order Methods\thanks{This paper
presents research results of the Belgian Network DYSCO (Dynamical Systems, Control, and Optimization), funded by the Interuniversity Attraction Poles Programme, initiated by the Belgian State, Science Policy Office, and of the Concerted Research Action (ARC) programme supported by the Federation Wallonia-Brussels (contract ARC 14/19-060). The scientific responsibility rests with its authors. A.B.T. is a F.R.I.A. fellow.}
}

\titlerunning{Smooth Strongly Convex Interpolation and Exact Worst-case Performance of First-order Methods}    

\author{Adrien B.~Taylor~\and Julien M.~Hendrickx~\and Fran\c cois Glineur
}

\authorrunning{A.B. Taylor, J.M. Hendrickx, F. Glineur}

\institute{A.B. Taylor, J.M. Hendrickx, F. Glineur \at
              Universit\'e catholique de Louvain, ICTEAM Institute/CORE, B-1348 Louvain-la-Neuve, Belgium\\
              E-mail: Adrien.Taylor@uclouvain.be; Julien.Hendrickx@uclouvain.be; Francois.Glineur@uclouvain.be           
}

\date{Date of current version: \today}

\maketitle

\begin{abstract}
We show that the exact worst-case performance of fixed-step
first-order methods for {unconstrained optimization of} smooth (possibly strongly) convex functions
can be obtained by solving convex programs.

Finding the worst-case performance of a black-box first-order method
is formulated as an optimization problem over a set of smooth (strongly)
convex functions and initial conditions. We develop closed-form necessary and
sufficient conditions for smooth (strongly) convex interpolation,
which provide a finite representation for those functions. This
allows us to reformulate the worst-case performance estimation problem
as an equivalent finite dimension-independent semidefinite
optimization problem, whose exact solution can be recovered up to numerical
precision. Optimal solutions to this performance estimation problem
provide both worst-case performance bounds and explicit functions
matching them, as our smooth (strongly) convex interpolation procedure is
constructive.

Our works build on those of Drori and Teboulle
in {[Math.\@ Prog.\@ 145 (1-2), 2014]} who introduced and solved relaxations of the
performance estimation problem for smooth convex functions.

We apply our approach to different fixed-step first-order methods with
several performance criteria, including objective function accuracy and
gradient norm. We conjecture several numerically supported worst-case
bounds on the performance of the {fixed-step} gradient, fast gradient and optimized {gradient} methods, both in the smooth convex and the smooth strongly
convex cases, and deduce tight estimates of the optimal step size for
the gradient method.
\end{abstract}

\section{Introduction to performance estimation}
\label{sec:intro}
Consider the standard unconstrained minimization problem $$\min_{x\in\mathbb{R}^d} f(x),$$ where $f$ is a smooth convex function, possibly strongly convex. First-order black-box methods, which only rely on the computation of $f$ and its gradient at a sequence of iterates, can be designed to solve this type of problem iteratively. A central question is then to estimate the accuracy of  solutions computed by such a method. More precisely, given a class of problems and a first-order method, one wishes to establish the worst-case accuracy of solutions that can obtained after applying a given number of iterations, {i.e.,} the performance of the method on the given class of problems.

Many first-order algorithms have been proposed in the literature for smooth convex or smooth strongly convex functions, for which one usually provides a theoretical \emph{upper bound} on the worst-case accuracy after a number of iterations ({see e.g.,}~\cite{Book:Nesterov} or~\cite[Chap.\@6]{Book:Bertsekas} for recent overviews).
{However,} many analyses focus on the asymptotic rate of convergence of these bounds, rather than trying to compute exact numerical values. Similarly, \emph{lower bounds} on the performance of first-order black-box methods on given classes of problems can be found in the literature (see {e.g.,} the seminal~\cite{Book:NemirovskyYudin}), again often with a focus on asymptotic rates of convergence. In many situations, the asymptotic rate of {convergence}  
of the best available methods match those lower bounds.

Nevertheless, the \emph{exact} numerical value of the worst-case performance of a given method is usually unknown. This {is because} upper bounds are not assessed precisely, {i.e.,} are known only up to a (possibly unspecified) constant. Another reason is that lower bounds for specific methods are not very frequently developed, and that general lower bounds (valid for all methods) can be quite weak for specific methods, especially if those methods do not feature the best possible  asymptotic rate of convergence. Finally, even if exact numerical values are known  for both lower and upper bounds, and share the same (optimal) asymptotic rate of convergence, a significant gap between the numerical values of those lower and upper bounds can subsist. If one cares about the worst-case efficiency of a first-order method in practice, this gap can translate into a very large uncertainty on the concrete behavior of a method.

This work is not concerned with asymptotic rates of convergence. It will focus on the  computation of the exact worst-case performance of a given first-order black-box method, on a given class of functions, after a given number of iterations. We prove that this question can be formulated and solved exactly as a (finite-dimensional) convex optimization problem, with the following attractive features:
\begin{itemize} \item Our formulation is a semidefinite optimization problem whose dimension is proportional to the {square of the} number of iterations of the method to be analyzed.
\item Any dual feasible solution of our formulation provides an upper bound on the worst-case performance. This solution can be {easily converted} into a standard proof establishing a bound on the performance ({i.e.,} a series of valid inequalities).
\item Any primal feasible solution of our formulation provides a lower bound on the worst case performance. This solution can be easily converted into a concrete function on which the method exhibits the corresponding performance.
\item Hence our formulation is exact, {i.e.,} its optimal value provides the exact worst-case performance.
\end{itemize} 
Our formulation covers both smooth convex functions and smooth strongly convex functions in a unified fashion. It covers a very large class of first-order methods which includes the majority of standard methods {for smooth unconstrained convex optimization}. It can be applied to a variety of performance measures, such as objective function accuracy, gradient norm, or distance to an optimal solution.

\subsection{Formal definition}
Our goal is to express the worst-case performance of an optimization algorithm as the solution of an optimization problem. This approach was pioneered by Drori and Teboulle~\cite{Article:Drori}, who called it a {p}erformance {e}stimation {p}roblem (PEP). We now provide a formal definition for this problem.

We consider unconstrained minimization problems involving a given class of objective functions, and only treat first-order black-box methods. This means that the method can only gather information about the objective function using an oracle $\mathcal{O}_f$, which returns first-order information about specific points, {i.e.,} $\mathcal{O}_f(x) = \{ f(x), {\nabla f(x)} \}$. Formally, the first $N$ iterates generated by a first-order black-box method $\mathcal{M}$ (which correspond to $N$ calls of the oracle), starting from an initial point $x_0$, can be described with 
\begin{align}
\begin{split} \label{eq:MO}
x_1& = \mathcal{M}_1\left(x_0,\mathcal{O}_f(x_0)\right),  \\
x_2& = \mathcal{M}_2\left(x_0,\mathcal{O}_f(x_0),\mathcal{O}_f(x_1)\right),\\
&\vdots\\
x_N& = \mathcal{M}_N\left(x_0,\mathcal{O}_f(x_0),\hdots,\mathcal{O}_f(x_{N-1})\right).
\end{split}
\end{align}

In order to measure the performance of a given method $\mathcal{M}$ on a specific function $f$ with a specific starting point, we introduce a performance criterion $\mathcal{P}$ to be minimized, that will only depend on the function $f$ and the sequence of the iterates $\{x_0, x_1, \ldots, x_N\}$ generated by the method. Since we are in a black-box setting, we require that the criterion can be computed from the output of the oracle $\mathcal{O}_f$, which has only access to the iterates as well as to an additional point $x_*$, defined to be any minimizer of function $f$ (the latter being necessary if the criterion has to compare iterates to an optimal solution).

Examples of this performance criterion $\mathcal{P}(\mathcal{O}_f,x_0,\hdots,x_N,x_*) $ include the objective function accuracy $f(x_N)-f(x_*)$, the norm of the gradient $\| \nabla f(x_N) \|$, or the distance to an optimal solution $\| x_N - x_* \|$ (see also Section \ref{subsec:numerics_mingrad} for an example of criterion that does not only depend on the last iterate).

Finally, we consider a given class  $\mathcal{F}$ of smooth convex or smooth strongly convex functions, over which we wish {to} estimate the worst-case performance of a method after $N$ iterations. {Note that the dimension of functions belonging to class $\mathcal{F}$ is left unspecified; in particular we will allow class $\mathcal{F}$ to contain functions with varying dimensions, in order to obtain dimension-free results.}

As methods try to minimize the performance criterion, their worst-case performance is obtained by maximizing $\mathcal{P}$ over functions in $\mathcal{F}$, which can be written as

\begin{align}
w(\mathcal{F}, R, \mathcal{M}, N, \mathcal{P})=\sup_{f,x_0,{\ldots},x_N,x_*} \ &\mathcal{P}(\mathcal{O}_f,x_0,\hdots,x_N,x_*)  \tag{PEP}\label{Intro:PEP} \\
\text{ {such that} }& f\in \mathcal{F} \notag\\
& x^* \text{ is optimal for } f,\notag\\
& x_1, {\ldots}, x_N \text{ is generated {from $x_0$ by method $\mathcal{M}$ with $\mathcal{O}_f$}}, \notag\\
& \| x_0-x_*\|_2 \leq R. \notag
\end{align}

{P}arameter $R$ was introduced to bound the distance between the initial point $x_0$ and the optimal solution $x_*$. Indeed, it is well-known that in most situations, performance of a first-order method cannot be sensibly assessed without such a constraint (see also the discussion of Section \ref{subsec:homegeneity}).

\subsection{Finite-dimensional reformulation using interpolation}

Because it involves an unknown function $f$ as a variable, problem \eqref{Intro:PEP} is infinite-dimensional. Nevertheless, using the black-box property of the method (and of the performance criterion), we will show that a completely equivalent finite-dimensional problem can readily be formulated by restricting {the} variable $f$ to the knowledge of the output of its oracle $\mathcal{O}_f$ on the iterates $\{x_0, x_1, \ldots, x_N\}$ and $x_*$. {Indeed, denoting the output of the oracle at each iterate $x_i$ by $\mathcal{O}_f(x_i) = \{f_i, g_i\}$, method $\mathcal{M}$ defined by \eqref{eq:MO} can be equivalently rewritten as 
 \begin{align}
\begin{split} \label{eq:Mxg}
x_1& = \mathcal{M}_1\left(x_0, f_0, g_0 \right),  \\
x_2& = \mathcal{M}_2\left(x_0, f_0, g_0, f_1, g_1 \right),\\
&\vdots\\
x_N& = \mathcal{M}_N\left(x_0, f_0, g_0, \ldots, f_{N-1}, g_{N-1} \right).
\end{split}
\end{align}}{Now}, defining a set $I = \{ 0, 1, 2, \ldots, N, * \}$ for the indices of {the} iterates, 
we can reformulate \eqref{Intro:PEP} into a problem involving only the iterates $\{x_i\}_{i \in I}$, their function values $\{f_i \}_{i \in I}$ and their gradients~$\{g_i\}_{i \in I}$ as (using equivalence between optimality of $x_*$ and constraint $g_* = 0${, as our problem is unconstrained})
\begin{align*}
w^f(\mathcal{F}, R, \mathcal{M}, N, \mathcal{P})=\sup_{\left\{x_i,g_i,f_i\right\}_{i\in I}} \ &\mathcal{P}\left(\left\{x_i,g_i,f_i\right\}_{i\in I}\right), \tag{f-PEP}\label{Intro:dPEP} \\
\text{ {such that} } 
&\text{there exists $f \in \mathcal{F}$ such that } \mathcal{O}_f(x_i) = \{f_i, g_i\} \ \forall i \in I, \\
& g_* = 0, \\
& x_1, {\ldots}, x_N \text{ is generated {from $x_0$ by method $\mathcal{M}$ with $\left\{f_i,g_i \right\}_{i\in \{ 0, \ldots, N-1 \} }$}}, \notag\\
&\|x_0-x_*\|_2 \leq R. \end{align*}
The crucial part of this reformulation is the first constraint, which can be understood as requiring that the set of variables $\left\{x_i,g_i,f_i\right\}_{i\in I}$ can be \emph{interpolated} by a function belonging to the class $\mathcal{F}$. {This optimization problem is strictly equivalent to the original~\eqref{Intro:PEP} in terms of optimal value, since  every solution to~\eqref{Intro:dPEP} can be interpolated by a solution of~\eqref{Intro:PEP} and, reciprocally, every solution of~\eqref{Intro:PEP} can be discretized to provide a solution to~\eqref{Intro:dPEP}. From that it is clear that $w(\mathcal{F}, R, \mathcal{M}, N, \mathcal{P})=w^f(\mathcal{F}, R, \mathcal{M}, N, \mathcal{P})$.}

\subsection{Paper organization and main contributions} 

We focus on  the class of smooth (strongly) convex functions. Therefore an exact formulation of problem \eqref{Intro:PEP}  as \eqref{Intro:dPEP} will {require} a set of necessary and sufficient conditions for the existence of a smooth strongly convex interpolating function, which is the main result obtained in Section~\ref{sec:smoothcvxinterp}. This set of conditions, which is of independent interest, was previously only known for general nonsmooth convex functions. Our approach is fully constructive, as we also exhibit a procedure to interpolate a smooth (strongly) convex function from a set of points with their associated gradients and function values, when such an interpolating function exists.

In Section~\ref{sec:Peps}, we show  how the resulting finite-dimensional \eqref{Intro:dPEP} problem can be reformulated exactly into a (convex) semidefinite optimization problem, which provides the first tractable and provably exact formulation of the performance estimation problem. We allow consideration of both smooth convex and smooth strongly convex functions, as well as a large class of performance criteria.

Section~\ref{sec:numerics} then tests our approach numerically on several standard first-order methods, including the constant-step gradient method, the fast gradient method and the optimized {gradient} method from \cite{kim2014optimized}. We are able to confirm several bounds appearing previously in~\cite{Article:Drori}, and to conjecture several new worst-case performance bounds, including bounds for strongly convex functions, and bounds on the gradient norm (either for the final iterate, or the smallest norm among all iterates). Another byproduct of our results is a tight estimate of the optimal step size for the gradient method on smooth convex and smooth strongly convex functions.

\subsection{Prior work}
Drori and Teboulle \cite{Article:Drori} were first to consider the notion of a performance estimation problem. They focus exclusively on the case of smooth convex functions equipped with the performance criterion $f(x_N)-f_*$, and introduce the idea of reducing \eqref{Intro:PEP} to a finite-dimensional problem involving only the iterates $x_i$, their gradients $g_i$ and function values $f_i$, along with an optimal point $x_*$ and optimal value $f_*$. 
They treat several standard first-order algorithms, namely, the standard fixed-step gradient algorithm, the heavy{-}ball method
and the accelerated gradient method~\cite{Nesterov:1983wy}. In their approach,  \eqref{Intro:PEP} is expressed as a non-convex {q}uadratic {m}atrix {p}rogram~\cite{beck2007quadratic}, which is then relaxed and dualized.  The resulting convex problem is then used to provide bounds on the worst-case performance (and, in some cases, is solved analytically).  As will be shown later in this paper (see  Section~\ref{sec:numerics}), because of the use of a relaxation  and the dualization of a non-convex problem,
these bounds are in general not tight, although they are in many special cases.

A Section in~\cite{Article:Drori}  is also devoted to the optimization of the coefficients of a {fixed-step} first-order black-box method. More precisely, a numerical optimization solver is used to identify a method performing  best according to their relaxation of the performance estimation problem. This approach is taken further in~\cite{kim2014optimized}, which provides an analytical description of this optimized method. Again we stress that,
due to the non-tightness of the relaxation in general, these optimized methods are not guaranteed to have the best possible performances.

Another computational approach for the analysis and design of first-order algorithms is proposed in~\cite{lessard2014analysis}, in which optimization procedures are regarded as dynamical systems. Integral {q}uadratic {c}onstraints (IQC), which are usually used to obtain stability guarantees on complicated dynamical systems, are adapted in order to obtain sufficient conditions for the convergence of optimization algorithms. This methodology {is able to establish iteration-independent linear rates of convergence by solving a single small semidefinite program. However those bounds{, valid} for any number of iteration{s, }{are in general not tight, i.e.,\@} more conservative than {ours and those} of \cite{Article:Drori} when used to estimate worst-case performance after a given finite number of iterations (see Subsection~\ref{sss:sc} for an example). In addition, while this methodology is well-suited for studying {the linear convergence rates of }algorithms for smooth strongly convex optimization, it fails to recover the exact {sublinear} rates in the {non-strongly convex case}. 
}

\section{Smooth strongly convex interpolation}
\label{sec:smoothcvxinterp}

This section develops a necessary and sufficient condition for the existence of a smooth strongly convex function interpolating through a given set of data triples $\left\{x_i,g_i,f_i\right\}_{i\in I}$, {i.e.,} deciding whether there exists a smooth strongly convex function $f$ such that $f(x_i) = f_i$ and $g_i \in \partial f(x_i)$ for all $i \in I$.

This result generalizes the well-known set of conditions guaranteeing the existence of a convex, possibly nonsmooth interpolating function (see Theorem~\ref{th:convexinterp} in Subsection~\ref{subsec:CvxInterp}). It is the main technical ingredient of our exact convex reformulation of performance estimation problems.

\subsection{Definitions and problem statement}
\label{subsec:defandprobstat}

We start by defining the functional class of interest, using the standard point of view from convex analysis --- we refer to classic books~\cite{bauschke2011convex,JBHU,Book:Rockafellar,rockafellar1998variational} for details. Given two parameters $\mu$ and $L$ satisfying $0 \le \mu < L \le +\infty$, we consider proper closed convex functions ({i.e.,} whose epigraph{s} are non-empty closed convex {sets}) satisfying both a smoothness condition (depending on the parameter $L$, which is the Lispchitz constant of the gradient) and a strong convexity condition (depending on the parameter $\mu$). We explicitly allow the case $L=+\infty$ {to include nonsmooth functions}, while $\mu$ on the other hand is always assumed to be finite.  In the rest of this paper, we use the  conventions $1/{+\infty}=0$ and $+\infty-\mu=+\infty$ {to deal with the case $L=+\infty$.}

\begin{definition} [$L$-smooth $\mu$-strongly convex functions]Consider a proper and closed convex  function $f:\mathbb{R}^d\rightarrow\mathbb{R}\cup\left\{+\infty\right\}$, and constants $\mu\in\mathbb{R}^+$, $L\in\mathbb{R}^{+}\cup\left\{+\infty\right\}$ with $\mu< L$. We say that function $f$ is $L$-smooth $\mu$-strongly convex (which we denote by $f\in\mathcal{F}_{\mu,L}{(\mathbb{R}^d)}$) if and only if the following two conditions are satisfied:
\begin{itemize}
\item[(a)] inequality $\frac{1}{L}\norm{ g_1 - g_2}_2\leq \norm{x_1-x_2}_2$ holds for all pairs  $x_1, x_2\in\mathbb{R}^d$ and corresponding subgradients ${g_1}, g_2 \in \mathbb{R}^d$ ({i.e.,} such that $g_1\in\partial f(x_1)$ and $g_2\in\partial f(x_2)$){;}
\item[(b)] function $f(x)-\frac{\mu}{2}\norm{x}_2^2$ is convex.
\end{itemize}
\label{def:smstrcvx}
\end{definition}

This definition is not entirely standard, as it involves subgradients and allows {the} constant $L$ to be equal to $+\infty$. 
In the case of a finite $L$, condition (a) immediately implies uniqueness of the subgradient at each point, hence {differentiability of} the function, and we recover the well-known Lipschitz condition on the gradient of a smooth function.
On the other hand, when $L=\infty$, condition (a) becomes vacuous, and the function can be non-differentiable.
Condition (b) can easily be seen to be algebraically equivalent to the standard definition of strong convexity (and the case $\mu=0$ corresponds to a convex but not strongly convex function). The class of proper closed convex functions simply corresponds to $\mathcal{F}_{0,\infty}$. The case $L=\mu$ can be safely discarded, as it only involves simple quadratic functions whose minimization is trivial. The reason for this slightly non-standard definition of $\mathcal{F}_{\mu,L}$ and the possibility of choosing $L=+\infty$ will become clear later, when dealing with {Legendre-}Fenchel conjugation.

As  explained above and in the introduction, our approach to express the original infinite{-}dimensional \eqref{Intro:PEP} in a finite-dimensional fashion relies on an interpolating condition for smooth strongly convex functions. This directly  motivates the following definition.

\begin{definition}[$\mathcal{F}_{\mu,L}$-interpolation] Let $I$ be an index set, and consider the set of triples \\$S = \left\{(x_i,g_i,f_i)\right\}_{i\in I}$ where  $x_i,g_i\in\mathbb{R}^d$ and $f_i\in\mathbb{R}$ for all $i \in I$. Set $S$ is  \emph{$\mathcal{F}_{\mu,L}$-interpolable} if and only if there exists a function $f\in\mathcal{F}_{\mu,L}{(\mathbb{R}^d)}$ such that we have both $g_i\in\partial f(x_i)$ and $f(x_i)=f_i$ for all $i\in I$. 
\label{def:CvxComp}
\end{definition}

\subsection{Necessity and sufficiency of conditions for smooth convex interpolation}
\label{subsec:ce}

Our goal is to identify a set of necessary of sufficient conditions involving the set of data triples and characterizing the existence of an interpolating function. Finding necessary conditions is relatively easy: starting from any set of necessary conditions that holds on the whole domain of a smooth strongly convex function, one can simply restrict this set to those conditions involving only points $x_i$ with $i \in I$ ({i.e.,} to discretize it). For example, it is well-known that the class of $L$-smooth convex functions $\mathcal{F}_{0,L}{(\mathbb{R}^d)}$ is  characterized by the pair of inequalities
\begin{align}
&f(y) \geq f(z) + \nabla f(z)^{\top\!}(y-z), \quad\quad \forall \ y,z \in \mathbb{R}^d, \label{eq:c1} \tag{C1} \\
&||\nabla f(y) - \nabla f(z) ||_2 \leq L ||y-z||_2, \quad\quad \forall \ y,z \in \mathbb{R}^d.&\notag
\end{align}
Therefore, specializing those conditions for $y=x_i$ and $z=x_j$ with $i,j \in I$ leads to the following set of inequalities, which is \emph{necessary}  for the existence of an interpolating function in $\mathcal{F}_{0,L}${:}
\begin{align}
&f_i\geq f_j+g_j^{{\top\!}}(x_i-x_j), \quad\quad \forall  i,j\in I,\label{eq:xce1}\tag{C1f}\\
&||g_i-g_j||_2 \leq L ||x_i-x_j||_2,\quad\quad \forall  i,j\in I .\notag
\end{align}
Now, perhaps surprisingly, it turns out that this latter set of conditions is \emph{not sufficient} to guarantee { $\mathcal{F}_{0,L}$-interpolability}, despite the fact that the originating conditions \eqref{eq:c1} \emph{are sufficient} to guarantee that $f \in \mathcal{F}_{0,L}{(\mathbb{R}^d)}$. In order to see that, consider the following {example with $I=\{1,2\}$ and $d=1$}:
\[ (x_1,g_1,f_1)=(-1,-2,1) \text{ and } (x_2,g_2,f_2)=(0,-1,0).
\]
This pair can clearly not be interpolated by a function {in $\mathcal{F}_{0,L} (\mathbb{R}^1)$ for any $L$}, as there is an unavoidable non-differentiability at $x_1$ {as illustrated on~\figref{fig:counter_ex_cvx_int}}. However, it satisfies Conditions~\eqref{eq:xce1} with $L=1$, which is therefore  not sufficient to guarantee smooth convex interpolation. 
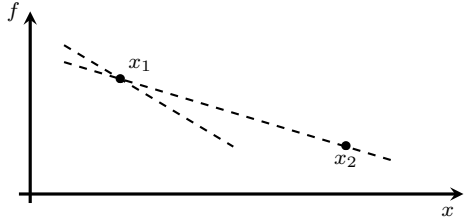
\begin{figure}
\begin{center}
\begin{tikzpicture}[yscale=0.3,xscale=0.3]
\tikzset{fleche/.style={->, >=latex, thick},
	    flecheB/.style={<->, >=latex, very thick},
 	    droite/.style={thick, dashed}};
 	    
\draw[very thick, black, -stealth] (-8.5,-0.1) -- (11.2,-0.1);
\draw (10.5,-0.3) node[below]{$x$};
\draw[very thick, black, -stealth] (-8,-0.5) -- (-8,8);
\draw (-8.1,8) node[left]{$f$};

\node (A) at (-4,5) {$\bullet$};
\node (B) at (6,2) {$\bullet$};
\node (Agrad1) at (-6.5,5+6/4) {};
\node (Agrad2) at (1,5-12/4) {};
\node (Bgrad1) at (-6.5,5+3/4) {};
\node (Bgrad2) at (8,2-6/10) {};
\draw (A) node[above right] {$x_1$};
\draw (B) node[below] {$x_2$};
\draw [droite] (Bgrad1.center) to (Bgrad2.center);
\draw [droite] (Agrad1.center) to (Agrad2.center);
\node (A) at (-4,5) {$\bullet$};
\node (B) at (6,2) {$\bullet$};
\end{tikzpicture}
\end{center}
\caption{{Example $(x_1,g_1,f_1)=(-1,-2,1)$ and $(x_2,g_2,f_2)=(0,-1,0)$ {for $I=\{1,2\}$ and $d=1$}. This set {satifies conditions ~\eqref{eq:xce1} but} cannot be interpolated by a smooth convex function: {the convexity requirement forces the interpolating convex function to lie entirely above its linear under-approximations, which lead to an unavoidable non-differentiability at $x_1$}.}}
\label{fig:counter_ex_cvx_int}
\end{figure}

Similarly, we can carry out the same exercise for the following conditions, also well-known to be equivalent to inclusion on $\mathcal{F}_{0,L}{(\mathbb{R}^d)}$ when imposed on the whole space:
\begin{align}
&f_i\geq f_j+g_j^{{\top\!}}(x_i-x_j), \quad \quad \forall\ i,j\in I,\label{eq:ce2}\tag{C2f}\\
&f_i\leq f_j+g_j^{{\top\!}}(x_i-x_j)+\frac{L}{2}||x_i-x_j||^2_2,\quad \quad \forall\ i,j\in I. \notag
\end{align}
With an appropriate use of an additional dimension {($d=2$)}, one can readily observe that some information may be hidden to this pair of inequalities{.} {Consider the example}
\[ (x_1,g_1,f_1)=\left(\begin{pmatrix}
0\\0
\end{pmatrix}, \begin{pmatrix}
1\\0
\end{pmatrix},0 \right) \text{ and }
(x_2,g_2,f_2)=\left(\begin{pmatrix}
1\\0
\end{pmatrix}, \begin{pmatrix}
1\\ 1
\end{pmatrix},1 \right),
\]
from which no smooth convex interpolation can be made (again, unavoidable non-differentiability at both $x_1$ and $x_2$). However, both Conditions~\ref{eq:xce1} and~\ref{eq:ce2} are satisfied with $L=1$.

Those examples illustrate the weakness of a naive approach that consists in discretizing standard necessary and sufficient conditions defined on the whole space. If those discretized conditions were used in a performance estimation problem over a given class of functions $\mathcal{F}_{\mu,L}$, they would implicitly allow the performance of functions that do not belong to the class $\mathcal{F}_{\mu,L}$ to be taken into account. This would correspond to a relaxation of the original performance estimation problem, and would only lead to upper bounds on the worst-case performance. {To conclude this section, note that {any set of necessary and sufficient conditions} for smooth convex interpolability {must be} the discretization of some necessary and sufficient conditions on the whole domain, whereas the previous examples precisely show that the converse is not true.}

In the next subsections, we follow a more principled approach in order to tackle the $\mathcal{F}_{\mu,L}$-interpolation problem. We start with a special case of convex interpolation, that of proper convex functions without smoothness or strong convexity requirement ({i.e.,} the class $\mathcal{F}_{0,\infty}$), for which a solution is well-known.

\subsection{Convex interpolation}
\label{subsec:CvxInterp}

In order to build interpolation conditions for the class of smooth strongly convex functions, we begin by constructing interpolation conditions for the simpler class of convex functions $\mathcal{F}_{0,\infty}{(\mathbb{R}^d)}$. As this result will be one of building blocks of the smooth strongly convex interpolation procedure, a simple constructive proof of this theorem is provided.

\begin{theorem}[Convex interpolation] \label{th:convexinterp} {The s}et $\left\{(x_i,g_i,f_i)\right\}_{i\in I}$ is $\mathcal{F}_{0,\infty}$-interpolable if and only if
\begin{equation}
f_i \geq f_j + g_j^{\!\top}  ( x_i-x_j)\quad \forall i,j\in I.
\label{eqn:Thm:CvxCompl1}
\end{equation}
\label{thm:CvxComp}
\end{theorem}
\begin{proof}
(Necessity.) Assume there exists a convex function $f: \mathbb{R}^d \rightarrow \mathbb{R}$ such that $f_i=f(x_i)$ and $g_i\in\partial f(x_i) \ \forall i\in I$. {The d}efinition of {a} subgradient then immediately implies that \[ f_i \geq f_j + g_j^{\!\top}  (x_i-x_j) \quad \forall i,j\in I.\]
(Sufficiency.) Define the following piecewise-linear convex function \[ f(x)=\max_{j\in I} \left\{ f_j + g_j^{\!\top}  ( x-x_j)\right\}.\]
Since $f$ is the pointwise maximum of a finite number of affine functions, its epigraph is a non-empty polyhedron, {and} hence $f$ is convex, closed and proper. In addition, $f(x_i)=f_i$ holds by construction. Indeed, we first see that 
\begin{align*}
f_i&=f_i+g_i^{\!\top}  (x_i-x_i){,}\\
&\leq \max_{j\in I} \left\{ f_j + g_j^{\!\top}  ( x_i-x_j)\right\}=f(x_i).
\end{align*}
Therefore, we have $f_i\leq f(x_i)$. In addition to this, we have 
\begin{align*}
f(x_i)&=\max_{j\in I} \left\{ f_j + g_j^{\!\top}  ( x_i-x_j)\right\},\\
&\leq f_i \quad \text{using Condition~\eqref{eqn:Thm:CvxCompl1} for each } j,
\end{align*}
which allows to conclude that $f(x_i)=f_i$.
The construction also implies that $g_i\in\partial f(x_i)${, i.e.,}
\begin{align*}
f(x)&=\max_{j\in I} \left\{ f_j + g_j^{\!\top}  (x-x_j)\right\} \quad \forall x\in\mathbb{R}^d,\\
&\geq f_i + g_i^{\!\top}  (x-x_i) \quad \forall i\in I, x\in\mathbb{R}^d,\\
&\geq f(x_i) + g_i^{\!\top}  (x-x_i) \quad \forall i\in I, x\in\mathbb{R}^d.
\end{align*}\qed
\end{proof}
One should note that the effective domain {of $f$} is $\dom f=\mathbb{R}^d$ --- that is, the function takes finite values for all $x\in~\mathbb{R}^d$. This is of course not the only way of reconstructing a valid $f$. For example, we could choose
\begin{align*}
f(x)=\left\{\begin{array}{ll}
\max_j \left\{ f_j + g_j^{\!\top}  ( x-x_j)\right\} \quad \quad &\text{if }x\in\text{conv}\left(\left\{x_i\right\}_{i\in I}\right),\\
+\infty & \text{otherwise.}
\end{array}\right.
\end{align*}

\begin{remark}Our interpolation problem is an extension of the classical finite convex integration problem, which is concerned with {the recovery of a} convex function from a set of points $x_i$ associated with a subgradient $g_i$ ({i.e.,}  function values are not specified). Finite convex integration is treated in details in~\cite{article:FiniteConvexIntegration} (only in the convex case $\mu=0$ and $L=+\infty$). {It} is the finite version of the continuous convex integrability problem, which is treated in~\cite{Book:Rockafellar}.

{A direct necessary and sufficient condition for deciding whether a set is convex integrable is to require the existence of function values $f_i$ for which the set $\left\{(x_i,g_i,f_i)\right\}_{i\in I}$ is convex interpolable. It is however also possible to derive a set of inequalities that does not involve unknown function values $f_i$, using the so-called cyclic monotonicity conditions. More precisely, consider for every sequence $\{ i_1, i_2, \ldots, i_N \}$ of distinct indices in $I$, and the corresponding cyclic sequence of successive pairs $\{ (i_1,i_2), (i_2, i_3), \ldots, (i_{N-1},i_N), (i_N,i_1)\}$. Summing inequality \eqref{eqn:Thm:CvxCompl1} over each pair of indices in this cyclic sequence produces a necessary inequality that does not involve function values $f_i$. Moreover, the set of all such inequalities, obtained from all possible sequences of distinct indices, is necessary and sufficient for convex integration, see e.g. \cite{article:FiniteConvexIntegration,Book:Rockafellar}. Note that this condition features a much larger number of inequalities.}
\end{remark}

\subsection{Conjugation and minimal curvature subtraction}
\label{subsec:PropSmStrCvx}
In this section, we review some concepts and results needed for our generalization of convex interpolation to smooth convex interpolation. We begin with the concept of \emph{conjugation} operation.  This operation is a key element in our approach, since it provides a way to reduce the general smooth strongly convex interpolation problem to a simpler convex interpolation problem.
\begin{definition} Given a {proper} function $f: \mathbb{R}^d\rightarrow\mathbb{R}\cup\left\{+ \infty \right\}$, the (Legendre-Fenchel) conjugate $f^*:\mathbb{R}^d\rightarrow\mathbb{R}\cup\left\{+ \infty \right\}$ of $f$ is defined as:
$$f^*(y)=\sup_{x\in \mathbb{R}^d} y^{\!\top}x-f(x).$$
\end{definition}
Conjugate functions enjoy numerous useful properties. Among other, they are always closed and convex. In fact, conjugation realizes a one-to-one correspondence in the set of proper closed convex functions ({i.e.,}\@ an involution), {{see for example }Theorem 12.2 of \cite{Book:Rockafellar}. That is, for any $f\in\mathcal{F}_{0,\infty}{(\mathbb{R}^d)}$ we have $f^{*}\in\mathcal{F}_{0,\infty}{(\mathbb{R}^d)}$ and $f^{**}=f$.}

{Among the useful properties of conjugate functions, we note that for any function $f\in\mathcal{F}_{0,\infty}{(\mathbb{R}^d)}$, conjugation can be seen as an operation reversing the roles of the coordinates and the subgradients: any subgradient (resp. coordinate) in one space becomes a coordinate (resp. subgradient) in the second space. In other words, for any function $f\in\mathcal{F}_{0,\infty}{(\mathbb{R}^d)}$ {and its conjugate $f^*$}, it is equivalent to {require that $x$ and $g$ satisfy {condition} $g\in\partial f(x)$, {condition} $x\in\partial f^*(g)$ {or} condition} $f(x) + f^*(g) = g^{\!\top}  x$. This can be obtained using first-order optimality conditions on the definition of conjugate function; we refer to Theorem~23.5 from \cite{Book:Rockafellar} for more details.} The next theorem emphasizes the effect of this link for the class of smooth convex functions.
\begin{theorem} Consider a function $f\in\mathcal{F}_{0,\infty}{(\mathbb{R}^d)}$. We have $f\in\mathcal{F}_{0,L}{(\mathbb{R}^d)}$ if and only if $f^*\in\mathcal{F}_{1/L,\infty}{(\mathbb{R}^d)}$.
\label{thm:ConjStrCvxLLipsch}
\end{theorem}
Theorem~\ref{thm:ConjStrCvxLLipsch} is basically Proposition~12.60 of~\cite{rockafellar1998variational} in the case $L<+\infty$ and reduces to {Theorem 12.2 of~\cite{Book:Rockafellar}} in the case $L=+\infty$. This also explains why we need to include the case {$L=+\infty$} in our interpolation problem: this is so that we can include the conjugates of smooth but non{-}strongly convex functions in $\mathcal{F}_{0,L}{(\mathbb{R}^d)}$.

The next lemma gives a simple way of expressing smooth strongly convex functions in terms of smooth functions. 

\begin{theorem}
Consider a function $f\in\mathcal{F}_{0,\infty}{(\mathbb{R}^d)}${. We have $f\in\mathcal{F}_{\mu,L}{(\mathbb{R}^d)}$ if and only if $f(x) - \frac{\mu}{2} \norm{x}_2^2 \in \mathcal{F}_{0,L-\mu}{(\mathbb{R}^d)}$.}
\label{lem:StrCvxEq}
\end{theorem}
This theorem holds true by Definition~\ref{def:smstrcvx} when $L=+\infty$. The case $L<+\infty$ can be found in the proof of Theorem~2.1.11 in~\cite{Book:Nesterov}.

\subsection{Necessary and sufficient conditions for smooth strongly convex interpolation}
\label{subsec:GeneralSmoothStrCvxInterp}

We now focus on transforming the smooth strongly convex interpolation problem into a convex interpolation problem. In order to do so, we mainly use the two previously defined operations: conjugation (using Theorem~\ref{thm:ConjStrCvxLLipsch}) and minimal curvature subtraction (using Theorem~\ref{lem:StrCvxEq}). The reasoning is the following:
\begin{itemize}
\item[(i)] Reformulate the $\mathcal{F}_{\mu,L}$ interpolation problem into a $\mathcal{F}_{0,L-\mu}$ interpolation problem using minimal curvature subtraction.
\item[(ii)] Write the $\mathcal{F}_{0,L-\mu}$ interpolation problem into a $\mathcal{F}_{1/(L-\mu),\infty}$ interpolation problem using Legendre-Fenchel conjugation.
\item[(iii)] Transform the $\mathcal{F}_{1/(L-\mu),\infty}$ interpolation problem into a $\mathcal{F}_{0,\infty}$ interpolation problem using again minimal curvature subtraction.
\end{itemize}
The effect of minimal curv{atur}e subtraction on our interpolation problem, used in steps (i) and (iii), is described by the following Lemma.
\begin{lemma} Consider a finite set $\left\{(x_i,g_i,f_i)\right\}_{i\in I}$ with $x_i,g_i\in\mathbb{R}^d$ and $f_i\in\mathbb{R}$. The following propositions are equivalent for any constants $0\leq\mu<L\leq +\infty$:
\begin{itemize}
\item[(a)] $\left\{\left(x_i,g_i,f_i\right)\right\}_{i\in I}$ is $\mathcal{F}_{\mu,L}$-interpolable,
\item[(b)] $\left\{\left(x_i,g_i-\mu x_i,f_i-\frac{\mu}{2}\norm{x_i}_2^2\right)\right\}_{i\in I}$ is $\mathcal{F}_{0,L-\mu}$-interpolable.
\end{itemize}
\label{thm:eqCondInterp2}
\end{lemma}
\begin{proof}
{}[$(a)\Rightarrow(b)$] It follows from Theorem~\ref{lem:StrCvxEq} that if there exists $f\in\mathcal{F}_{\mu,L}{(\mathbb{R}^d)}$ interpolating the set, {then $h(x)=f(x)-\frac{\mu}{2}\norm{x}_2^2$} exists and satisfies $h\in\mathcal{F}_{0,L-\mu}{(\mathbb{R}^d)}$ and $\forall i\in I$:
$$h(x_i)=f_i-\frac{\mu}{2}\norm{x_i}^2_2, \quad g_i-\mu x_i\in\partial h(x_i).$$
Hence, the function $h\in\mathcal{F}_{0,L-\mu}{(\mathbb{R}^d)}$ interpolates the set $\left\{\left(x_i,g_i-\mu x_i,f_i-\frac{\mu}{2}\norm{x_i}^2_2\right)\right\}_{i\in I}$.

[$(a)\Leftarrow(b)$] If such a $h\in\mathcal{F}_{0,L-\mu}{(\mathbb{R}^d)}$ exists and satisfies the interpolation conditions $(b)$, then one can reconstruct a function $f(x)=h(x)+\frac{\mu}{2}\norm{x}_2^2$, $f\in\mathcal{F}_{\mu,L}{(\mathbb{R}^d)}$ which interpolates the set $\left\{\left(x_i,g_i,f_i\right)\right\}_{i\in I}$.\qed
\end{proof}
The effect of conjugation in step (ii) of the reduction procedure is precisely {described} in the following lemma.
\begin{lemma}Consider a finite set $\left\{(x_i,g_i,f_i)\right\}_{i\in I}$ with $x_i,g_i\in\mathbb{R}^d$ and $f_i\in\mathbb{R}$. The following propositions are equivalent for any constant $0<L\leq+\infty$:
\begin{itemize}
\item[(a)] $\left\{\left(x_i,g_i,f_i\right)\right\}_{i\in I}$ is $\mathcal{F}_{0,L}$-interpolable,
\item[(b)] $\left\{\left(g_i, x_i,x_i^{\!\top}  g_i-f_i\right)\right\}_{i\in I}$ is $\mathcal{F}_{1/L,\infty}$-interpolable.
\end{itemize}
\label{thm:eqCondInterp1}
\end{lemma}
\begin{proof}
{}[$(a)\Rightarrow(b)$] It follows from Theorem~\ref{thm:ConjStrCvxLLipsch} that if there exists $f\in\mathcal{F}_{0,L}{(\mathbb{R}^d)}$ then $f^*$ exists and satisfies $f^*\in\mathcal{F}_{1/L,\infty}{(\mathbb{R}^d)}$. In addition to that, if both $f\in\mathcal{F}_{0,L}{(\mathbb{R}^d)}$ and $f^*$ exists, then they satisfy $\forall i\in I$ the three conditions ({e.g., Theorem~23.5 of \cite{Book:Rockafellar}}):
$$ f(x_i) + f^*(g_i) = g_i^{\!\top}  x_i,\quad g_i\in\partial f(x_i), \quad x_i\in\partial f^*(g_i).$$

[$(b)\Rightarrow(a)$] If a function $f^*\in\mathcal{F}_{1/L,\infty}{(\mathbb{R}^d)}$ exists and satisfies the interpolation conditions (b), then the conjugate $f^{**}$ (which is convex, proper and closed by construction) satisfies $f^{**}\in\mathcal{F}_{0,L}{(\mathbb{R}^d)}$ by Theorem~\ref{thm:ConjStrCvxLLipsch}, {as well as} the interpolation conditions ({e.g., Theorem~23.5 of \cite{Book:Rockafellar}}) $\forall i\in I$: 
$$ f^{**}(x_i) + f^*(g_i) = g_i^{\!\top}  x_i,\quad g_i\in\partial f^{**}(x_i), \quad x_i\in\partial f^*(g_i).$$
We obtain the desired result by choosing $f=f^{**}$.\qed
\end{proof}
We are now properly armed in order to define all interpolation equivalences. Let us now use steps (i), (ii) and (iii) to prove the main theorem of this section.

\begin{theorem}[$\mathcal{F}_{\mu,L}$-interpolability] Set $\left\{(x_i,g_i,f_i)\right\}_{i\in I}$ is $\mathcal{F}_{\mu,L}$-interpolable if and only if the following set of conditions holds for every pair of indices $i \in I$ and $j \in I$
\begin{align}
f_i - f_j - g_j^{\!\top}  (x_i-x_j) \geq \frac{1}{2(1-\mu/L)}\left( \frac{1}{L}\norm{g_i-g_j}_2^2 + \mu \norm{x_i-x_j}_2^2 - 2\frac{\mu}{L} (g_j-g_i)^{\!\top}  (x_j-x_i)\right).\label{eq:Cond_Lmu_interp}
\end{align}
\label{thm:gencvxcomp}
\end{theorem}

\begin{proof}
We begin by showing the equivalence between the following propositions:
\begin{itemize}
\item[(a)] $\left\{\left(x_i,g_i,f_i\right)\right\}_{i\in I}$ is $\mathcal{F}_{\mu,L}$-interpolable,\vspace{0.3cm}
\item[(b)] $\left\{\left(x_i,g_i-\mu x_i,f_i-\frac{\mu}{2}\norm{x_i}_2^2\right)\right\}_{i\in I}$ is $\mathcal{F}_{0,L-\mu}$-interpolable,\vspace{0.3cm}
\item[(c)] $\left\{\left(g_i-\mu x_i,x_i,x_i^{\!\top}  g_i-f_i-\frac{\mu}{2}\norm{x_i}_2^2\right)\right\}_{i\in I}$ is $\mathcal{F}_{1/(L-\mu),\infty}$-interpolable,\vspace{0.3cm}
\item[(d)] $\left\{\left(g_i-\mu x_i,\frac{Lx_i}{L-\mu}-\frac{g_i}{L-\mu},\frac{Lx_i^{\!\top}  g_i}{L-\mu}-f_i-\frac{L\mu\norm{x_i}_2^2}{2(L-\mu)}-\frac{\norm{g_i}_2^2}{2(L-\mu)}\right)\right\}_{i\in I}$ is $\mathcal{F}_{0,\infty}$-interpolable,\vspace{0.3cm}
\item[(e)] $\left\{\left(\frac{Lx_i}{L-\mu}-\frac{g_i}{L-\mu}, g_i-\mu x_i, \frac{\mu x_i^{\!\top}  g_i}{L-\mu}+f_i-\frac{L\mu \norm{x_i}_2^2}{2(L-\mu)}- \frac{\norm{g_i}_2^2}{2(L-\mu)}\right)\right\}_{i\in I}$ is $\mathcal{F}_{0,\infty}$-interpolable.
\end{itemize}
Both $(a)\Leftrightarrow(b)$ and $(c)\Leftrightarrow(d)$ are direct application{s} of Lemma~\ref{thm:eqCondInterp2}, whereas both $(b)\Leftrightarrow(c)$ and $(d)\Leftrightarrow(e)$ are direct application{s} of Lemma~\ref{thm:eqCondInterp1}.
Theorem~\ref{thm:gencvxcomp} follows from equivalence between propositions (a) and (e) applied to the necessary and sufficient conditions for convex interpolation of Theorem~\ref{thm:CvxComp}. Finally, it is straightforward to check that condition (e) reduces to the statement of the {t}heorem. 
\qed
\end{proof}
\begin{remark}Note that one can also easily construct an interpolating function $f(x)$ for the original set of points from Theorem~\ref{thm:gencvxcomp}(a). It follows from Theorem~\ref{thm:CvxComp} that a possible interpolating function for the set $\left\{\left(\tilde{x}_i,\tilde{g}_i,\tilde{f}_i\right)\right\}_{i\in I}$ of Theorem~\ref{thm:gencvxcomp}(c) is given by
$$h(\tilde{x})=\max_i \left\{\tilde{f}_i + \tilde{g}^{\top\!}_{i}(\tilde{x}-\tilde{x}_i) + \frac{1}{2(L-\mu)}\norm{\tilde{x}-\tilde{x}_i}_2^2\right\}=\max_i  \ h_i(\tilde{x}).$$
This can be conjugated into an interpolating function $h^*(x)$ of the set given by Theorem~\ref{thm:gencvxcomp}(b). Using Theorem~16.5 from~\cite{Book:Rockafellar}, this can equivalently be written in the form
$$ h^*(x)=\text{conv}\left(h_i^*(x)\right),$$
where the $\mathrm{conv}$ operation takes the convex hull of the epigraphs of {the} $h_i^*$'s. Hence an interpolating function for the original set $\left\{\left(x_i,g_i,f_i\right)\right\}_{i\in I}$ is  given by
$$ f(x)=\text{conv}\left(h_i^*(x)\right)+\frac{\mu}{2}\norm{x}_2^2.$$
\end{remark}
{We provide an example of such an interpolat{ing} function on \figref{Fig:exInterp}.}

\begin{figure}[!ht]
\begin{center}
\begin{tabular}{cc}
   \includegraphics[scale=0.35]{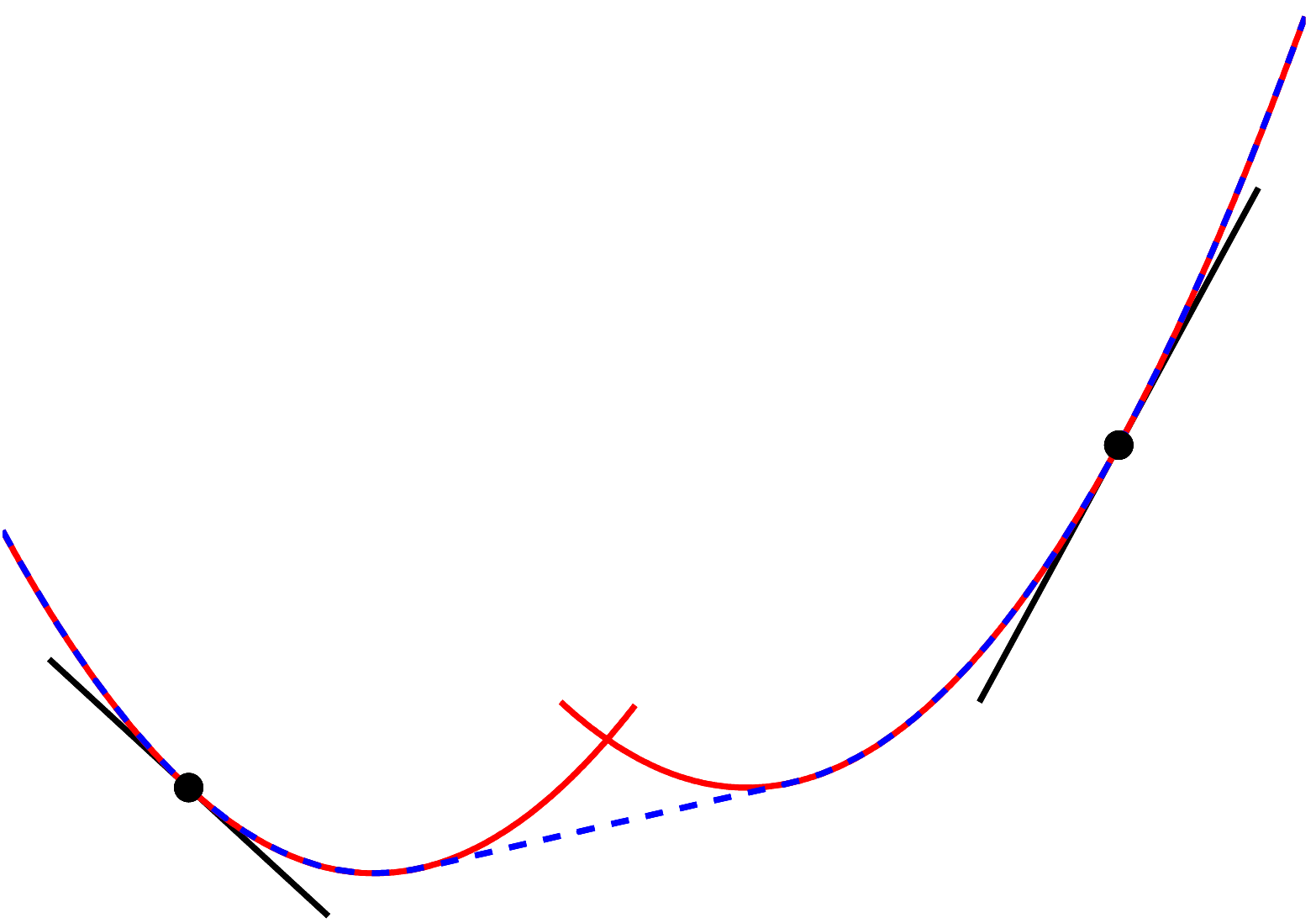} & \includegraphics[scale=0.35]{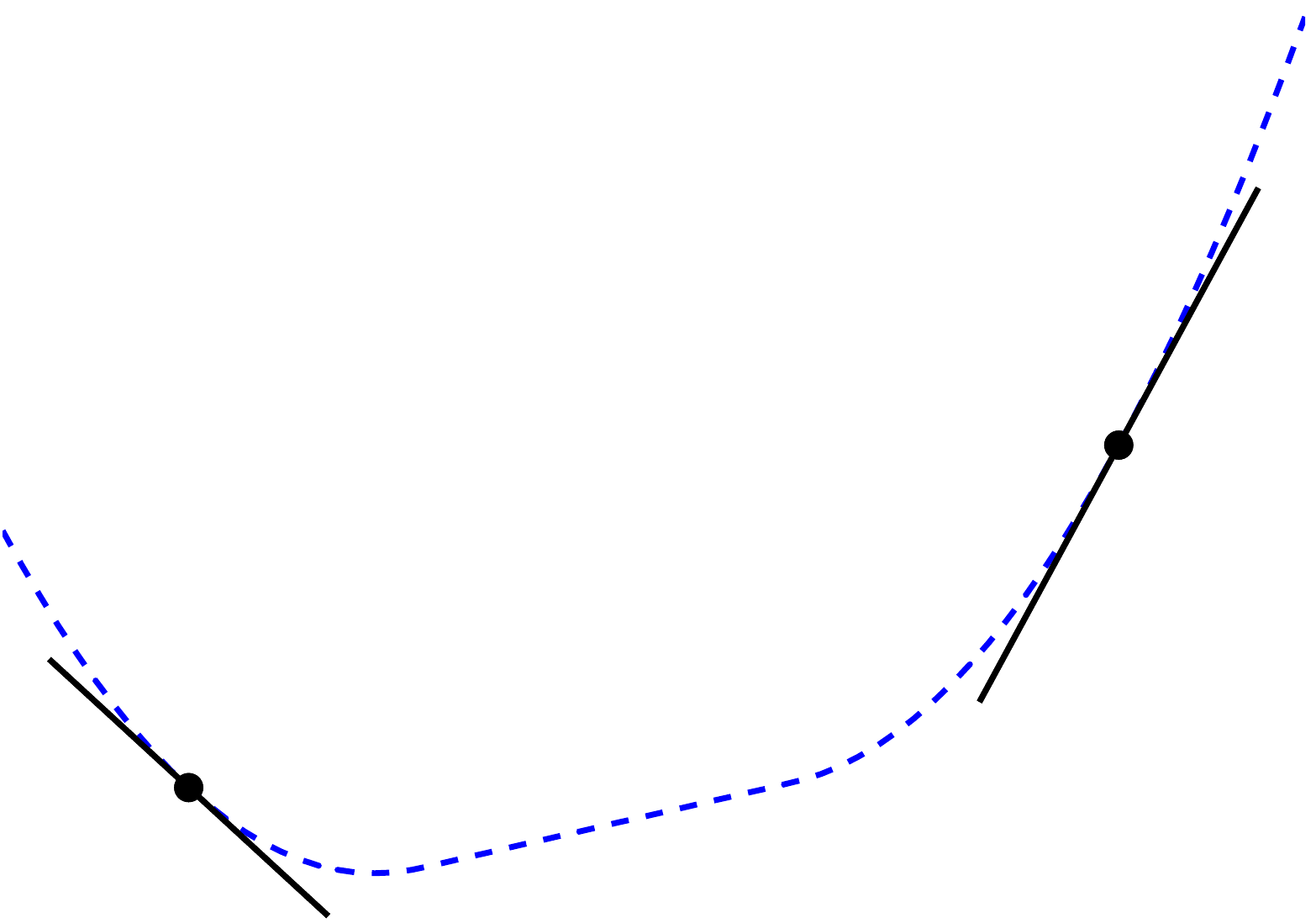}
\end{tabular}
\caption{{Example of an interpolat{ing} function; the {data triples} to be interpolated by a $1$-smooth convex function {are} $(x_1,g_1,f_1)=(2,2,3)$ and $(x_2,g_2,f_2)=(-3,-1,1)$. {Figure shows t}he upper{-}bounding quadratic functions $h_i^*(x)$ (red, left), the interpolat{ing} function $f(x)=\text{conv}\left(h_i^*(x)\right)$ (blue) and the gradients (black {tangents}).}} \label{Fig:exInterp}
\end{center}
\end{figure}
{It is straightforward to establish the equivalent interpolation conditions for both the smooth but non-strongly convex case ($\mu=0$) and the nonsmooth strongly convex case ($L=+\infty$).} In the first case --- given by Corollary~\ref{corr:Linterp} --- we find the discrete version of the well-known inequality characterizing $L$-smooth convex functions, which turns out to be necessary and sufficient{:} $$f(x)\geq f(y)+\nabla f^{\top\!}(y)(x-y)+\frac{1}{2L}\norm{\nabla f(x) - \nabla f(y)}_2^2.$$
\begin{corollary} The finite set $\left\{(x_i,g_i,f_i)\right\}_{i\in I}$ is $\mathcal{F}_{0,L}$-interpolable if and only if
\begin{align*}
f_i \geq f_j + g_j^{\!\top}  (x_i-x_j) + \frac{1}{2L}\norm{g_i-g_j}_2^2, \quad \quad \forall i,j\in  I.
\end{align*}
\label{corr:Linterp}
\end{corollary}
{Nonsmooth s}trongly convex interpolation conditions are given in Corollary~\ref{corr:muinterp}, which corresponds to the well-known inequality characterizing the subgradients of strongly convex functions.
\begin{corollary} The finite set $\left\{(x_i,g_i,f_i)\right\}_{i\in I}$ is $\mathcal{F}_{\mu,\infty}$-interpolable if and only if
\begin{align*}
f_i \geq f_j + g_j^{\!\top}  (x_i-x_j) + \frac{\mu}{2}\norm{x_i-x_j}_2^2, \quad \quad \forall i,j\in  I.
\end{align*}
\label{corr:muinterp}
\end{corollary}
\begin{remark}
{Following the spirit of Remark 1, we note that Theorem~\ref{thm:gencvxcomp} {can readily} be extended to handle the finite (and continuous) integration problems for $L$-smooth $\mu$-strongly convex functions (i.e., interpolation without function values). Indeed, summing  inequality  \eqref{eq:Cond_Lmu_interp} from Theorem~\ref{thm:gencvxcomp} over each pair in the cyclic sequence $\{ (i_1,i_2), (i_2, i_3), \ldots, (i_{N-1},i_N), (i_N,i_1)\}$ produces a necessary inequality that does not  involve function values $f_i$. Moreover one can show that the set of those inequalities for all possible sequences is necessary and sufficient for finite convex integration of $L$-smooth $\mu$-strongly convex functions, generalizing the standard cyclic monotonicity conditions. As an illustration, note that the following inequality
\[
 (g_i-g_j)^{\top\!}(x_i-x_j)\geq \frac{1}{1+\mu/L}\left( \frac{1}{L}\norm{g_i-g_j}^2_2+\mu \norm{x_i-x_j}_2^2\right),
\]
that is standard in the analysis of gradient methods on smooth strongly convex functions (see {e.g.,} Theorem 2.1.12 of~\cite{Book:Nesterov}), corresponds to cycles of length $2$ (i.e., cyclic sequences $\{ (i,j), (j,i) \}$).
The set of all such inequalities is necessary but not sufficient, as it omits longer cycles.}
\end{remark}

\section{A convex formulation for performance estimation}
\label{sec:Peps}
As explained in the introduction, our performance estimation problem can now be expressed in terms of the iterates and optimal point $\left\{x_i,g_i,f_i\right\}_{i\in \left\{0,{\ldots},N,*\right\}}$ only, using the interpolation conditions given by Theorem~\ref{thm:gencvxcomp}.

As our class of functions $\mathcal{F}_{\mu,L}{(\mathbb{R}^d)}$ and the first-order methods we study are invariant with respect to both an additive shift in the function values and a translation in their domain, we can assume without loss of generality {that} $x_*=0$ and $f_*=0$, which will simplify our derivations. We {can} also assume $g_*=0$, from {the} optimality conditions {of unconstrained optimization}. The problem can now be stated in its finite{-}dimensional formulation:

\begin{align*}
w{_{\mu,L}^{(d)}} (R, \mathcal{M}, N, \mathcal{P})=\sup_{\left\{x_i,g_i,f_i\right\}_{i\in I} {\in \bigl(\mathbb{R}^{d} \times \mathbb{R}^{d} \times \mathbb{R} \bigr)^{N+2}}}  & \mathcal{P}\left(\left\{x_i,g_i,f_i\right\}_{i\in I}\right), \tag{d-PEP}\label{Intro:dPEP2} \\[0.3cm]
\text{ {such that} } & \left\{x_i,g_i,f_i\right\}_{i\in I} \text{is $\mathcal{F}_{\mu,L}$-interpolable},\\[0.2cm]
& x_1, {\ldots}, x_N \text{ is generated {from $x_0$ by method $\mathcal{M}$ with \eqref{eq:Mxg}}}, \\[0.2cm]
& {\{x_*, g_*, f_*\} = \{0^d, 0^d, 0 \} \text{ and }} \norm{x_0 - x_*}_2 \leq R.
\end{align*}

{Problem \eqref{Intro:dPEP2} is an instance of \eqref{Intro:dPEP} where the function class $\mathcal{F}$ is chosen to be  $\mathcal{F}_{\mu,L}(\mathbb{R}^d)$, the set of  $d$-dimensional $L$-smooth $\mu$-strongly convex  functions, hence we have $w(\mathcal{F}_{\mu,L}(\mathbb{R}^d), R, \mathcal{M}, N, \mathcal{P})  = w{_{\mu,L}^{(d)}}(R, \mathcal{M}, N, \mathcal{P})$. Interestingly, in most situations of interest, quantity $w{_{\mu,L}^{(d)}}(R, \mathcal{M}, N, \mathcal{P})$ is monotonically increasing with $d$, as a higher dimensional function can usually mimic a lower dimensional one (see  Theorem~\ref{thm:sdp_pep} and subsequent results for a discussion on finite convergence of this sequence).}

{Finally, note that problem \eqref{Intro:dPEP2}} is  not convex, as it involves several non-convex quadratic constraints {({e.g.,}\@ $g_j^{\top\!} x_i$ terms in the interpolation conditions)}. In the next section, we show how~\eqref{Intro:dPEP2} can be cast as a convex semidefinite program~\cite{Vandenberghe94semidefiniteprogramming} when dealing with a certain class of first-order black-box methods, those {based on} fixed steps.

\subsection{Fixed-step first-order methods}
We hereby restrict ourselves to the class of fixed-step first-order methods{, where each iterate is obtained by adding a series of gradients steps with fixed step sizes to the starting point $x_0$.
\begin{definition}\label{def:GenFirstOrder} 
A method $\mathcal{M}$ is called a \emph{fixed-step method} if its iterates are computed according to
\[ x_i=x_0-\sum_{k=0}^{i-1} h_{i,k} g_k.\]
with fixed scalar coefficients $h_{i,k}$.
\end{definition}
}
{A fixed-step method performing $N$ steps is completely defined by the lower triangular $N \times N$ matrix $H = \left\{h_{i,k}\right\}_{1 \le i \le N, 0 \le k \le N-1}$ (where $h_{i,k}$ is defined to be zero if $k \ge i$).}
Note that many classical methods such as {the} gradient method {with constant step size} (GM) and {the} fast gradient method (FGM) are included in this class of algorithms (see the details in Section~\ref{sec:numerics}).

\subsection{A convex reformulation using a Gram matrix}

In order to obtain a convex formulation for~\eqref{Intro:dPEP2}, we {introduce} a Gram matrix {to describe the iterates and their gradients}. {Denoting}
\[ P=[g_0 \ g_1 \ \hdots \ g_N \ x_0] \]
{we define the symmetric $(N+2) \times (N+2)$ Gram matrix $G=P^{\top\!}P \in \mathbb{S}^{N+2}$, which is equivalent to \[ G = \left\{  G_{i,j} \right\}_{0 \le i, j \le N}  \text{ with }\left\{ \begin{array}{rl}  G_{i,j} =  g_i^\top g_j  & \text{ for any $0 \le i, j \le N$},  \\ 
G_{N+1,j} = x_0^\top g_j  & \text{ for any $0 \le j \le N$}, \\
G_{i,N+1} = g_i^\top x_0  & \text{ for any $0 \le i \le N$}, \\
G_{N+1,N+1} = x_0^\top x_0 & \end{array} \right.  \]}
{(note that the size of this matrix does not depend on the dimension of iterate $x_0$ and gradients $g_i$).

The constraints in problem \eqref{Intro:dPEP2} can now be entirely formulated in terms of the entries of the Gram matrix $G$ along with the function values $f_i$. Indeed all iterates apart from $x_0$ can be substituted out of the formulation using Definition~\ref{def:GenFirstOrder} of a fixed-step method, and the resulting formulation only involves function values $f_i$ and inner products between $x_0$ and all gradients $g_i$.

Note that the initial iterate $x_0$ and successive gradients $g_i$ of any solution to problem \eqref{Intro:dPEP2} can be transformed into a symmetric and positive semidefinite Gram matrix $G$. Moreover, since vectors $x_0$ and $g_i$ belong to $\mathbb{R}^d$, matrix $G$ has rank at most $d$. In the other direction, it is easy to see that any symmetric and positive semidefinite Gram matrix $G$ of rank at most $d$ can be converted back (using Cholesky decomposition for example) into $N+2$ vectors $x_0 \in \mathbb{R}^d$ and $g_i  \in \mathbb{R}^d$ which describe the initial iterate and successive gradients of a $d$-dimensional function (this transformation is however not unique). From those observations we can anticipate that an equivalent formulation of \eqref{Intro:dPEP2}  will rely on imposing that $G$ is symmetric and positive semidefinite, which is a convex constraint and will naturally lead to  a semidefinite program.}

{\subsection{Exact worst-case performance of fixed-step first-order methods as a semidefinite program}}

For notational convenience, we define vectors $h_i\in\mathbb{R}^{N+2}$ for any $i$ between $0$ and $N$ and  $h_*\in\mathbb{R}^{N+2}$ as follows{
$$h_i^{\top\!}=[ -h_{i,0} \ -\!h_{i,1} \ \hdots \ \ -\!h_{i,i-1} \ 0 \ \hdots \ 0 \ 1],\quad  h_*^{\top\!}=[0 \ \hdots  \ 0],$$}{so that we have }$x_i=P h_i$.  In order to lighten the notations we also define $u_i=e_{i+1}\in\mathbb{R}^{N+2}$, the canonical basis vectors, and $u_*$ the vector of zeros. Using those notations, we rewrite the interpolation constraints \eqref{eq:Cond_Lmu_interp} coming from Theorem~\ref{thm:gencvxcomp} in the following form:
\begin{align*}
 f_i\geq f_j &+ \frac{L}{L-\mu} (u_j^{\top\!} G h_i-u_j^{\top\!} G h_j) + \frac{1}{2(L-\mu)}(u_i-u_j)^{\top\!}G(u_i-u_j)\\
 & + \frac{\mu}{L-\mu} (u_i^{\top\!} G h_j-u_i^{\top\!} G h_i) + \frac{L\mu}{2(L-\mu)} (h_i-h_j)^{\top\!}G(h_i-h_j), \quad \quad \text{ for all } i,j\in I.
\end{align*}
{We can equivalently formulate all constraints using the trace operator, and add the distance constraint $\norm{x_0 - x_*}_2 \leq R$ on the starting point as well as the positive semidefiniteness constraint for $G$. Defining matrices $A_{ij}$ and $A_R$ in the following way}
\begin{align*}
2A_{ij}=& \ \frac{L}{L-\mu}\left( u_{j}(h_i-h_j)^{\top\!}+(h_i-h_j)u_{j}^{\top\!}\right)+\frac{1}{L-\mu}(u_i-u_j)(u_i-u_j)^{\top\!}\\ & \ +\frac{\mu}{L-\mu}\left( u_{i}(h_j-h_i)^{\top\!}+(h_j-h_i)u_{i}^{\top\!}\right)+\frac{L\mu}{L-\mu}(h_i-h_j)(h_i-h_j)^{\top\!},  \quad {\text{ for all } i,j \in I,}\\
A_R=& \ u_{N+1} u_{N+1}^{\top\!}.
\end{align*}
{we obtain the following compact formulation for the feasible region that is \emph{linear} in its variables $f \in \mathbb{R}^{N+1}$ and $G \in \mathbb{S}^{N+2}$}
\begin{align*}
 f_j-f_i + \Tr{ G A_{ij}} &\leq 0, \quad \quad \text{ for all } i,j\in I,\\
 \Tr{G A_R} - R^2 &\leq 0,\\
 G&\succeq 0 \;.
\end{align*}
{From the discussion at the end of the previous section, it is easy to see that any $d$-dimensional function $f$ and starting point $x_0 \in \mathbb{R}^d$ produce a feasible solution $(f,G)$ where matrix $G$ has rank at most $d$. On the other hand, any feasible solution $(f,G)$ where $G$ has rank at most $d$ can be interpolated into a $d$-dimensional function $f\in\mathcal{F}_{\mu,L}{(\mathbb{R}^d)}$ and a starting point $x_0 \in \mathbb{R}^d$. Indeed, matrix $G = P^\top P$ provides $x_0 \in \mathbb{R}^d$ and $N+1$ successive  gradients $g_i \in \mathbb{R}^d$, while the other iterates $x_i$ derive from the definition of the method. Our interpolating conditions ensure that a function compatible with these data triples $\left\{x_i,g_i,f_i\right\}_{i\in I}$ exists.}

{Considering finally} the performance criterion $\mathcal{P}$, we observe that any concave semidefinite-representable function {in} $G$ and $f$ leads to a worst-case estimation problem that can be cast as a convex semidefinite optimization problem (see {e.g.,} \cite{ben2001lectures}){, and that such a criteria does not depend on the dimension $d$ of the functions}. In particular, linear functions of the entries of $f$ and $G$ are suitable. Classical performance criteria such as $f(x_N)-f_*$, $\norm{\nabla f(x_N)}_2^2$ and $\norm{x_N-x_*}_2^2$ are indeed covered by this formulation. We focus below on the case of a linear performance criterion, but note that other criteria can be useful (see for example a concave piecewise linear criteria used in Section~\ref{subsec:numerics_mingrad}). 

{We can now state the main result of this paper.}

{\begin{theorem} \label{thm:sdp_pep} Consider the class $\mathcal{F}_{\mu,L}{(\mathbb{R}^{d})}$ of $L$-smooth $\mu$-strongly convex functions with $0\leq \mu < L \le \infty$, a fixed-step first-order method that computes $N$ iterates according to matrix $H \in \mathbb{R}^{N \times N}$, and a performance criterion $\mathcal{P}_{b,C}(f,G) = b^\top f + \Tr{CG}$ that depends linearly on the function values at those iterates and quadratically on the iterates and their gradients ($b\in\mathbb{R}^{N+1}$ and $C\in\mathbb{S}^{N+2}$). 

\noindent If $N \le d-2$, the worst-case performance after $N$ iterations of method $H$ applied to some function in  $\mathcal{F}_{\mu,L}{(\mathbb{R}^{d})}$ is equal to the optimal value of the following semidefinite program
\begin{align}
w^{sdp}_{\mu,L}(R,H,N,b,C) =\sup_{G\in\mathbb{S}^{N+2}, f\in\mathbb{R}^{N+1}}  b^{{\top\!}} f + \Tr{CG} &\tag{sdp-PEP}\label{PEPS:Pep_sdp}\\
\text{ {such that} } f_j-f_i + \Tr{ G A_{ij}} &\leq 0, \quad \quad i,j\in I,\notag\\
 \Tr{G A_R} - R^2 &\leq 0,\notag\\
 G&\succeq 0, \notag
\end{align}
with matrices $A_{ij}$ and $A_R$ as defined above. In others words,
\[ w^{sdp}_{\mu,L}(R,H,N,b,C)= w^{(d)}_{\mu,L}(R, \mathcal{M}, N, \mathcal{P}_{b,C}) \text{ for any $d \ge N+2$.} \]
\end{theorem}

\begin{proof} 
{We have already shown the two-way correspondence between functions in $\mathcal{F}_{\mu,L}{(\mathbb{R}^{d})}$ and feasible solutions of this problem where matrix $G$ has rank at most $d$. Since matrix $G$ has size $N+2$, it has rank at most $N+2$, which establishes that this semidefinite program is a correct formulation of the performance estimation problem when $d \ge N+2$.\qed}
\end{proof}}

{The optimal value $w^{sdp}_{\mu,L}(R,H,N,b,C)$ of \eqref{PEPS:Pep_sdp} is not necessarily finite or attained at some feasible point. However, when $L$ is finite, any continuous performance criterion $\mathcal{P}$ will force the optimal value to be attained and finite, because the constraints on variables $f$ and $G$ force the domain to be compact.}{
\begin{proposition} Under the assumptions of Theorem~\ref{thm:sdp_pep}, the optimum of~\eqref{PEPS:Pep_sdp} is attained and finite when $L<\infty$.
\label{prop:opt_achieved}
\end{proposition}
\begin{proof}
To show that the solution of \eqref{PEPS:Pep_sdp} is attained and finite, it suffices to prove that its feasible region is compact (since the objective is continuous). We first prove that the iterates of method $H$ applied to any function in  $\mathcal{F}_{\mu,L}{(\mathbb{R}^d)}$ with $L<\infty$ are bounded, as well as their gradients.

Note that  the Lipschitz condition on the gradients \eqref{eq:xce1} with $j=*$ shows that if iterate $x_i$ is bounded, gradient $g_i$ is also bounded. We proceed by recurrence. We start with the fact that $x_0$ is bounded, using the assumption that $x_* = 0$ and constraint $\norm{x_0 - x_*}_2 \leq R$. This implies that $g_0$ is bounded, hence that $x_1$ is bounded using Definition~\ref{def:GenFirstOrder} of a fixed-step method. This implies in turn that $g_1$ is bounded, then that $x_2$ is bounded, and so on until we have shown that all iterates and gradients are bounded.

Condition \eqref{eq:ce2} with $j=*$ then implies that function values $f_i$ are bounded. Therefore all entries in variables $f$ and $G$ are bounded which, combined with closedness of the feasible region, establishes the claim. \qed
\end{proof}}
\begin{remark} When $L=+\infty$ {(recall the conventions $1/{+\infty}=0$ and $+\infty-\mu=+\infty$ used in this paper)}, the {feasible region may be} unbounded and it is possible to design feasible functions which drive standard performance criteria arbitrarily away from $0$. Nevertheless, performance estimation on such nonsmooth functions could still be tackled after introduction of another appropriate Lipschitz condition on the {class of} functions, such as $||g_i||_2\leq L${. We leave this as a topic} for further research and, in the rest of this text, restrict ourselves to the smooth case $L<+\infty$. \end{remark}

{Our formulation \eqref{PEPS:Pep_sdp} is dimension-independent, and computes the exact worst-case performance of a first-order method with $N$ steps as long as the class of functions of interest contains functions of dimension at least $N+2$. This corresponds to the so-called large-scale optimization setting, which is usually assumed when analyzing the worst-case of first-order methods.  Function classes with smaller dimensions can also be handled with the addition of a (non-convex) rank constraint on $G$. We obtain the following result, whose proof is straightforward.

\begin{proposition} Consider the setting of Theorem~\ref{thm:sdp_pep}. If $d < N+2$, the worst-case performance after $N$ iterations of method $H$ applied to some function in  $\mathcal{F}_{\mu,L}{(\mathbb{R}^{d})}$ is equal to the optimal value of the semidefinite program~\ref{PEPS:Pep_sdp} under the additional constraint $\rank G \le d$.
\end{proposition}
Note that this also establishes that the sequence $\{w^{(d)}_{\mu,L}(R, \mathcal{M}, N, \mathcal{P}_{b,C})\}_{d=1,2,\hdots}$ is monotonically increasing, and that it converges for a finite value of $d$.
\begin{corollary}
The worst-case performance after $N$ steps of a fixed-step method on a $L$-smooth ($\mu$-strongly) convex function is achieved by an $N+2$-dimensional function. \label{corr:nplus2}
\end{corollary}}

{Finally, note that, when applied to the gradient method in the non{-}strongly {convex} case  ($\mu=0$),  problem \eqref{PEPS:Pep_sdp} is equivalent to one of the formulations proposed in~\cite{Article:Drori}, more specifically to their problem (G). Theorem~\ref{thm:sdp_pep}  establishes that this relaxation is in fact exact.}

{\subsection{A dual semidefinite program to generate upper bounds}}
In general, it is not easy to find an analytical optimal solution to~\eqref{PEPS:Pep_sdp}. Hence, we are also interested in a generic and easier way of obtaining analytical upper bounds on the performance of a given algorithm. A classical way of doing so is to work with the Lagrangian dual of~\eqref{PEPS:Pep_sdp}:
\begin{align}
\label{dualPep} \tag{d-sdp-PEP} 
\inf_{\lambda_{ij}, \tau}  \ \tau R^2\ \text{{ such that }  } \tau A_R - C + \sum_{i,j\in I} \lambda_{ij} A_{ij}  &\succeq 0,\notag\\
  b-\sum_{i,j\in I} \lambda_{ij} (u_{j}-u_{i}) &= 0,\notag\\
 \lambda_{ij}&\geq 0, \quad \quad i,j\in I,\notag\\
 \tau &\geq 0, \notag
\end{align}
whose feasible solutions will provide theoretical upper bounds on the worst-case behavior of every fixed-step first-order method (using weak duality). Note that the final dual formulation used in~\cite{Article:Drori}, which deals with the case $\mu=0$, can be recovered by taking $\lambda_{ij}=0$ for $i+1\neq j$ or $i\neq *$ in our dual, {i.e.,}\@ it is a restriction of \eqref{dualPep} with a potentially larger optimal value.

The next {theorem} guarantees that no duality gap occurs between \eqref{PEPS:Pep_sdp} and \eqref{dualPep} under the technical assumption that $h_{i,i-1}\neq 0$ ($i\in\left\{1,\hdots,N\right\}$). This assumption is reasonable as it only implies {that, at each iteration, the most recent gradient obtained from the oracle has to be used in the computation of the next iterate}. The {theorem} will also guarantee the existence of a dual feasible point attaining the optimal value of the primal-dual pair of estimation problems \eqref{PEPS:Pep_sdp} and \eqref{dualPep}, {i.e.,} a tight upper bound on the worst-case performance of the considered method. 

\begin{theorem} \label{prop:main} The optimal value of the dual problem~(\ref{dualPep}) with $0\leq \mu < L < \infty$ is attained and equal to $w^{(sdp)}_{\mu,L}(R,H,N,b,c)$ under the assumption{s} that $h_{i,i-1}\neq 0$ for all $i\in\left\{1,\hdots,N\right\}$.
\end{theorem}
\begin{proof} We use the classical Slater condition~\cite{Book:Boyd} on the primal problem in order to guarantee a zero duality gap --- that is, we show that~\eqref{PEPS:Pep_sdp} has a feasible point with $G\succ 0$. The reasoning is divided in two parts; we consider first the case $\mu=0$ and $L=2+2\cos(\pi/(N+2))$, and we generalize it to general $\mu<L$ afterwards. Consider the quadratic function $f(x)=\frac{1}{2}x^{\!\top}  Qx$ with the following tridiagonal positive definite matrix $Q$ 
\[ Q=
\begin{pmatrix}
2 \ \ & 1 \ & 0 \ & 0 & \hdots & \ 0\\[0.12cm]
1 \ \ & 2 \ & 1 \ & 0 & \hdots & \ 0\\
0 \ \ & 1 \ & 2 \ & 1 & \ddots & \vdots\\
\vdots &  \ & \ddots & \ddots & \ddots & 
\end{pmatrix} \succ 0. \]
We show how to construct a full-rank $G$ feasible for~\eqref{PEPS:Pep_sdp} using the values of the quadratic function $f$. In order to do so, we exhibit a full-rank matrix 
$$P=[ x_0 \ g_0 \ g_1 \ \hdots g_N]$$
corresponding to the application of a given method (with $h_{i,i-1}\neq 0$) to the quadratic function $f$. Indeed, choosing $x_0=Re_1$, we can show that $P$ is upper triangular with non-zero diagonal entries. {Then we have
\begin{align*}
&g_0=Qx_0=2e_1+e_2,\\
&x_1=x_0-h_{1,0}g_0,\\
&g_1=Qx_1=g_0-h_{1,0}Qg_0=2e_1+e_2-h_{1,0}(4e_1+4e_2+e_3).\\
\end{align*}}Hence $g_1$ has a non-zero element associated with $e_3$ whereas the only non-zero elements of $g_0$ are associated with $e_1$ and $e_2$. Now, we assume that $g_{i-1}$ has a non-zero element corresponding to $e_{i+1}$ and zero elements corresponding to $e_{k}$ for all $k> i+1$, while all previous gradients have zero components corresponding to $e_{k}$ for all $k> i$. {Then we have}
$$g_i^{\top\!}e_{i+2}=x_i^{\top\!}Qe_{i+2}=x_i^{\top\!}(e_{i+1}+2e_{i+2}+e_{i+3}),$$
with $x_i^{\top\!}e_{i+2}=x_i^{\top\!}e_{i+3}=0$ and $x_i^{\top\!}e_{i+1}\neq0$ because of the recurrence assumption and the iterative form of the algorithm:{
\begin{align*}
& x_i^{\top\!}e_{i+1} = \underbrace{x_0^{\top\!}e_{i+1}}_{=0}- \sum^{i-2}_{k=0} h_{i,k}\underbrace{g_{k}^{\top\!}e_{i+1}}_{=0} - h_{i,i-1}\underbrace{g_{i-1}^{\top\!}e_{i+1}}_{\neq 0},\\
& x_i^{\top\!}e_{i+2} = \underbrace{x_0^{\top\!}e_{i+2}}_{=0}- \sum^{i-2}_{k=0} h_{i,k}\underbrace{g_{k}^{\top\!}e_{i+2}}_{=0} - h_{i,i-1}\underbrace{g_{i-1}^{\top\!}e_{i+2}}_{= 0},\\
& x_i^{\top\!}e_{i+3} = \underbrace{x_0^{\top\!}e_{i+3}}_{=0}- \sum^{i-2}_{k=0} h_{i,k}\underbrace{g_{k}^{\top\!}e_{i+3}}_{=0} - h_{i,i-1}\underbrace{g_{i-1}^{\top\!}e_{i+3}}_{= 0}.
\end{align*}}Hence, $g_i$ has a non-zero element associated with $e_{i+2}$. We {deduce} that the following components are {equal to} zero by computing $g_i^{\top\!}e_{i+2+k}$ for $k>0$:
$$g_i^{\top\!}e_{i+2+k}=x_i^{\top\!}Qe_{i+2+k}=x_i^{\top\!}(e_{i+1+k}+2e_{i+2+k}+e_{i+3+k}),$$
which is zero because of the algorithmic structure of $x_i${, i.e.,
\begin{align*}
& x_i^{\top\!}e_{i+1+k} = \underbrace{x_0^{\top\!}e_{i+1+k}}_{=0}- \sum^{i-2}_{k=0} h_{i,k}\underbrace{g_{k}^{\top\!}e_{i+1+k}}_{=0} - h_{i,i-1}\underbrace{g_{i-1}^{\top\!}e_{i+1+k}}_{= 0}.
\end{align*}}Hence{,} $P$ is an upper triangular matrix with positive entries on the diagonal, and is therefore full-rank. In order to make this statement hold for general $\mu<L$, observe that the structure of the matrix is preserved using the operation {($I_{N+2}$ is the identity matrix)}
$$ Q'=\left(Q-\lambda_{\min}(Q)I_{N+2}\right) \frac{(L-\mu)}{\lambda_{\max}(Q)-\lambda_{\min}(Q)} + \mu I_{N+2}.$$
The corresponding quadratic function is {easily seen to be} {$L$-smooth and} $\mu$-strongly convex.
Therefore, the interior of the domain of~\eqref{PEPS:Pep_sdp} is non-empty and Slater's condition
applies for  $\mu<L$, ensuring that no duality gap occurs and that the dual optimal value is attained.\qed
\end{proof}
{ One can note that Theorem~\ref{prop:main} guarantees the existence of a {fully explicit} proof (i.e., a combination of valid inequalities, or equivalently, a dual feasible solution) for any worst-case function (see the example at the end of this section).}
\subsection{Homogeneity of the optimal values with respect to $L$ and $R$}
\label{subsec:homegeneity}
{W}e observe that, for most performance criteria, one can predict {how the worst-case performance depends} from parameters $L$ and $R$, {provided the fixed step sizes contained in $H$ are scaled appropriately (i.e., inversely proportional to $L$). In the rest of this paper we will only consider such scaled (normalized) step sizes. Therefore, the corresponding performance estimation problems have only to be solved numerically in the case $R=1$ and $L=1$, from which a general bound valid for any $L$ and $R$ can be deduced.

More specifically, a classical reasoning involving appropriate scaling operations easily leads to the following 
 homogeneity relations} for the standard criteria $f(x_N)-f_*$, $\norm{\nabla f(x_N)}_2$ and $\norm{x_N-x_*}_2$:
\begin{align*}
&w^{sdp}_{\mu,L}(R,{H/L},N,f(x_N)-f_*)=LR^2 \ w^{sdp}_{\kappa,1}(1,H,N,f(x_N)-f_*),\\
&w^{sdp}_{\mu,L}(R,{H/L},N,\norm{\nabla f(x_N)}^2_2)={\bf \color{red} L^2R^2}\ w^{sdp}_{\kappa,1}(1,H,N,\norm{\nabla f(x_N)}^2_2),\\
&w^{sdp}_{\mu,L}(R,{H/L},N,\norm{x_N-x_*}^2_2)={\bf \color{red} R^2}\ w^{sdp}_{\kappa,1}(1,H,N,\norm{x_N-x_*}^2_2),
\end{align*}
where $\kappa=\mu/L$ is the inverse condition number {and $H/L$ {describes} the fixed-step method obtained by dividing all step sizes $h_{i,j}$ by the Lipschitz constant $L$.} {Results in the rest of this paper implicitly rely on these relations.}

{\subsection{A simple example}
Consider the very simple case of a method performing a single gradient step  using the non-standard step-size $\frac{3}{2L}$, i.e., $x_1 = x_0 - \frac{3}{2L} \nabla f(x_0)$ (this is actually conjectured to be the best possible step size for a single step, see Section~\ref{sss:Drori}). One wishes to estimate the worst-case objective function accuracy after taking that step, i.e., maximize $f(x_1) - f_*$, over all $L$-smooth convex functions. Solving the corresponding semidefinite formulation \eqref{PEPS:Pep_sdp} with $\mu=0$, $N=1$, $H=\begin{pmatrix} \frac32 \end{pmatrix}$ and $\mathcal{P}_{b,C}(f, G) = f_1$ provides the  optimal value \[ w^{sdp}_{0,L}\left(R, \begin{pmatrix} \frac32 \end{pmatrix}, 1,\begin{pmatrix} 0 \\ 1 \end{pmatrix} , 0^{3 \times 3} \right)=\frac{LR^2}{8}, \] attained by the following optimal solution \[ f_0 = \frac{LR^2}{2} , f_1 = \frac{LR^2}{8} \text{ and } G = L R^2 \begin{pmatrix} L  & - L /2 & 1 \\ - L /2 & L /4  & - 1/2 \\ 1 & -1/2 & 1/L  \end{pmatrix} \succeq 0. \]
This means that $f(x_1) - f_* \le \frac{LR^2}{8}$ holds for any $f \in \mathcal{F}_{0,L}(\mathbb{R}^d)$ for any $d$ and provided that $\| x_0 - x_* \| \le R$.  It is easy to check that function $f(x)= \frac{L}{2} x^2 \in \mathcal{F}_{0,L}(\mathbb{R})$ achieves this worst-case when started from $x_0 = R$. Indeed one can successively evaluate $f_0=f(x_0)=\frac{LR^2}{2}$, $g_0 = \nabla f(x_0) = LR$, $x_1 = R - \frac32 R = -\frac{R}{2}$, $f_1=f(x_1)=\frac{LR^2}{8}$ and $g_1 = -\frac{LR}{2}$. This function is one-dimensional since the optimal $G$ has rank one (note that Corollary~\ref{corr:nplus2} only guaranteed the existence of a three-dimensional worst-case).

Solving the dual problem \eqref{dualPep} leads to the same optimal value $\frac{LR^2}{8}$, attained by  optimal multipliers $\lambda_{01}=\lambda_{*0}=\lambda_{*1}=\frac12 \geq 0 $ and $\tau = \frac{L}{8}$. The corresponding dual slack matrix is \[ S= \frac12  \begin{pmatrix} 
1/L &  1/L & -1/2\\
1/L & 1/L & -1/2\\
-1/2 & -1/2 & L/4 \end{pmatrix} = \frac{L}{2} \begin{pmatrix} -1/L \\ -1/L \\ 1/2\end{pmatrix}  \begin{pmatrix} -1/L & -1/L & 1/2\end{pmatrix}  \succeq 0. \]
From this dual solution, a fully explicit proof of the worst-case performance can be derived, which can be checked independently without any knowledge about our approach. Indeed, linear equalities in the dual imply that the  objective accuracy $f(x_1)-f_*$ can be written exactly as follows
\begin{eqnarray*} f(x_1) - f(x_*) &=& \frac12 \left(f(x_1)-f(x_0)+\nabla f(x_1)^\top(x_0-x_1)+\frac{1}{2L}\norm{\nabla f(x_0)-\nabla f(x_1)}_2^2\right)  \\ & &
+ \frac12  \left(f(x_0)- f(x_*) + \nabla f(x_0)^\top (x_* -x_0) +\frac{1}{2L} \norm{\nabla f(x_0)-\nabla f(x_*)}_2^2 \right) 
\\ & & + \frac12  \left(f(x_1) - f(x_*)  + \nabla f(x_1)^\top (x_* - x_1) +\frac{1}{2L} \norm{\nabla f(x_1)-\nabla f(x_*)}_2^2\right)  \\
& & + \frac{L}{8}  \| x_0 - x_* \|^2  - \frac{L}{2} \left\| \frac{1}{2} (x_0-x_*)  - \frac{\nabla f(x_0)}{L} - \frac{\nabla f(x_1)}{L}  \right\|^2 
\end{eqnarray*}
(where for the last term we write the quadratic form $\Tr{S G}$ as a square, since $S$ is rank-one). This equality, which is straightforward to check using $x_1 = x_0 - \frac{3}{2L} \nabla f(x_0)$ and $\nabla f(x_*)=0$, immediately implies inequality $f(x_1) - f_* \le \frac{L }{8}\| x_0 - x_* \|^2$, since the first three bracketed expressions are nonpositive because of inequalites from Corollary~\ref{corr:Linterp} valid for all functions in $\mathcal{F}_{0,L}$.}

\section{Numerical performance estimation of standard first-order algorithms}
\label{sec:numerics}

In this section we apply the convex PEP formulation to study convergence of the fixed-step gradient method (GM), the standard fast gradient method (FGM) and the optimized gradient method (OGM) proposed by~\cite{kim2014optimized}. We begin {with} the GM for smooth convex optimization, whose worst-case is conjectured {in}~\cite{Article:Drori} to be attained on a simple one-dimensional function. Numerical experiments with our exact formulation confirm this conjecture. Further experiments on the worst-case complexity for different methods, problem classes and performance criteria lead to a series of conjectures based on worst-case functions possessing a similar shape. We conclude this section with the study of a nonlinear performance criteria corresponding to the smallest gradient norm among all iterates computed by the method.

All numerical results in this section were obtained on an Intel $3.5$Ghz desktop computer using a combination of the YALMIP modeling environment in MATLAB~\cite{Article:Yalmip}, the MOSEK~\cite{Article:Mosek} and SeDuMi ~\cite{Article:Sedumi} semidefinite solvers and the VSDP (verified semidefinite programming) toolbox~\cite{Techreport:vsdp}.

\subsection{Gradient method}
\subsubsection{A conjecture on smooth convex functions by Drori and Teboulle~\cite{Article:Drori}}\label{sss:Drori}
Consider the classical fixed-step gradient method (GM) with constant  step sizes applied to a smooth convex function in $\mathcal{F}_{0,L}{(\mathbb{R}^d)}$. Following the discussion in section \ref{subsec:homegeneity} we use normalized step sizes $\frac{h}{L}$, inversely proportional to $L$.

\begin{center}
\fbox{
  \parbox{0.65\textwidth}{
  \textbf{Gradient Method (GM)}
  \begin{itemize}
  \item[] Input: $f\in\mathcal{F}_{0,L}{(\mathbb{R}^d)}$, $x_0\in\mathbb{R}^d$, $y_0=x_0$.\\[-0.2cm]
  \item[] For $i=0:N-1$\\[-0.5cm]
      \begin{align*}
    x_{i+1}&=x_i-\frac{h}{L}\nabla f(x_i)
    \end{align*}
  \end{itemize}
  }
}
\end{center}
The following conjecture on the convergence of the worst-case objective function values was made in~\cite{Article:Drori}. 
\begin{conjecture}[\cite{Article:Drori}, Conjecture 3.1.] {Any sequence of iterates $\left\{x_i\right\}$ generated by the gradient method GM with constant normalized step size $0\leq h \leq 2$ on a smooth convex function} $f\in\mathcal{F}_{0,L}{(\mathbb{R}^d)}$ satisfies
$$f(x_N)-f_*\leq \frac{LR^2}{2}\max\left(\frac{1}{2Nh+1},(1-h)^{2N} \right).$$
\label{conj:GM_f_TD}
\end{conjecture}
A proof of the conjecture is provided in \cite{Article:Drori} for  step sizes $0\leq h \leq 1$, leaving the case $1 < h < 2$ open. We also recall that the upper bound in this conjecture cannot be improved, as it matches the performance of the GM on two specific one-dimensional functions. Indeed, define 
\begin{align*}
&f_1(x)=\left\{\begin{array}{ll}
\frac{LR}{2Nh+1}|x|-\frac{LR^2}{2(2Nh+1)^2}\quad  & \text{if } |x| \geq \frac{R}{2Nh+1},\\[0.1cm]
\frac{L}{2}x^2 & \text{else}{,}
\end{array}
\right.\\
&f_2(x)=\frac{L}{2}x^2.
\end{align*}
It is straightforward to check that the final objective value accuracy of GM on $f_1$ is equal to {$ \frac{LR^2}{2}\frac{1}{2 Nh+1}$, and that  it is equal to $\frac{LR^2}{2}(1-h)^{2N}$ on $f_2$}. This means that the conjecture can be reformulated as saying that the worst-case behavior of the GM according to objective function accuracy is achieved by function $f_1$ or $f_2$, depending on which of the two is worst (which will depend only on the normalized step size $h$ and number of iterations $N$).

Intuitively, the behavior of GM on piecewise affine-quadratic $f_1$ corresponds to a situation in which iterates slowly approach the optimal value without oscillating around it ({i.e.,} no overshooting), whereas GM applied on purely quadratic $f_2$ generates a sequence oscillating around the optimal point. Those behaviors are illustrated on \figref{Fig:GM_on_f1_f2}. We also note that iterates for $f_1$ {stay on the affine piece of the function, and even never come close to the quadratic piece}. Interestingly, the existence of a one-dimensional worst-case function with a simple affine-quadratic shape will also be observed for the other algorithms studied in this section, both in the smooth convex and in the smooth strongly convex cases.

\begin{figure}[!ht]
\begin{center}
\begin{tabular}{p{6cm} c}
   \includegraphics[scale=0.3]{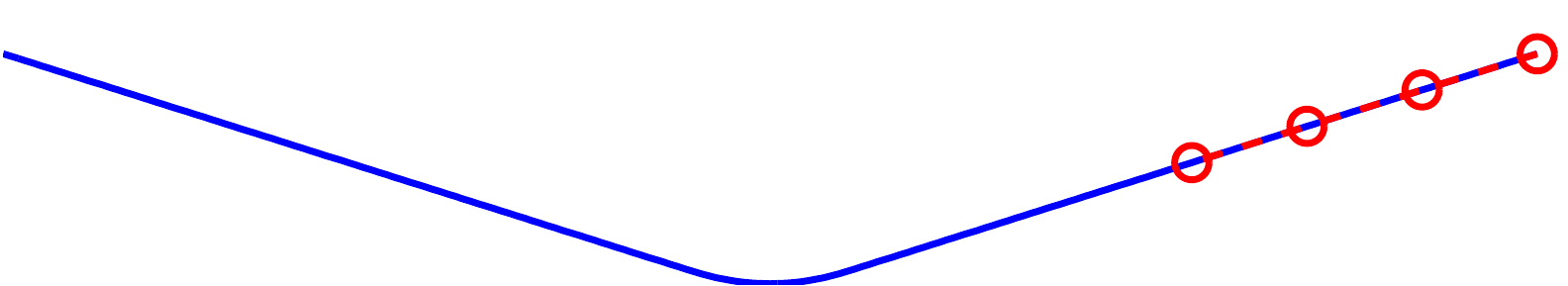} &
   \includegraphics[scale=0.3]{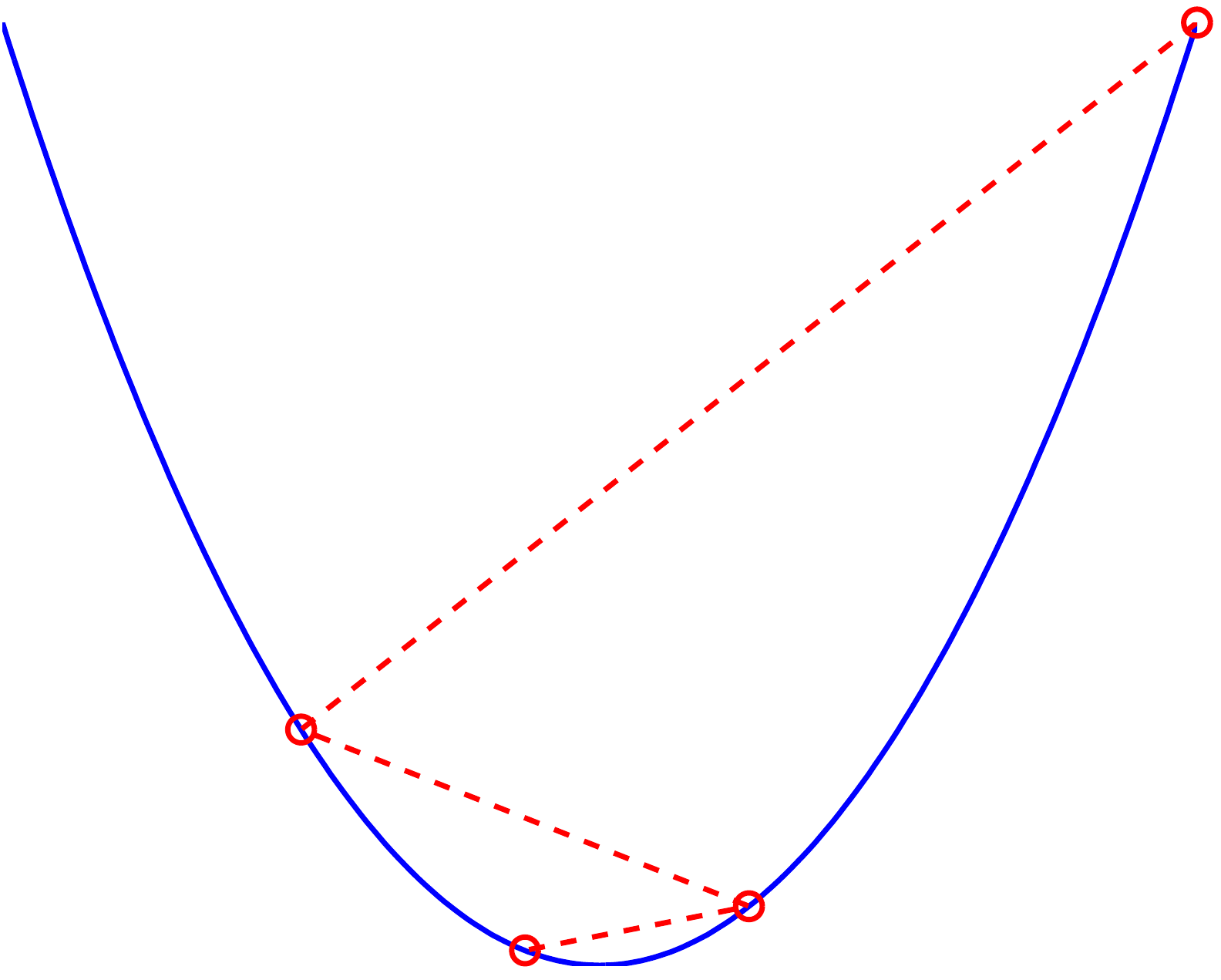} \\
\end{tabular}
\caption{Behavior of the gradient method on $f_1$ (left) and $f_2$ (right), for $L=R=1$. We observe that GM does not overshoot the optimal solution on $f_1$, while it does so at each iteration on $f_2$.} \label{Fig:GM_on_f1_f2}
\end{center}
\end{figure}

Empirical results from the numerical resolution of~\eqref{PEPS:Pep_sdp} strongly support Conjecture~\ref{conj:GM_f_TD}. Indeed, when comparing its predictions with numerically computed worst-case bounds, we obtained a maximal relative error of magnitude $10^{-7}$ (all pairs of values $N\in\left\{1,2,{\ldots},30\right\}$ and $h\in\left\{0.05,0.10,{\ldots},1.95\right\}$ were tested). {It is also worth pointing out that the Gram matrices computed numerically correspond to the one-dimensional worst-case functions $f_1$ and $f_2$ introduced above.}

Before going into the details of other methods, we underline another observation coming from~\cite{Article:Drori}: Conjecture~\ref{conj:GM_f_TD} also suggests the existence of an optimal step size $h_{\text{opt}}(N)$ for the GM --- optimal in the sense of achieving the lowest worst-case. That is, if you know in advance how many iterations of the GM you will perform, it suggests using a step size $h_{\text{opt}}(N)$ that is the unique minimizer of the right{-}hand side of the Conjecture~\ref{conj:GM_f_TD} for a fixed value of $N$. It is obtained by solving\footnote{This equation possesses several solutions, but the optimum is the unique point where the two terms feature derivatives of opposite signs (a necessary and sufficient condition for the maximum of two convex functions of one variable). This point can easily be computed numerically with an appropriate bisection method.
} the following non-linear equation in~$h_{\text{opt}}$ (for which no closed{-}form solution seems to be available):
 $$\frac{1}{2Nh_{\text{opt}}+1}=(1-h_{\text{opt}})^{2N}.$$
This optimal step size can be interpreted in terms of the trade-off between what we obtain on {functions} $f_1$ and $f_2$. On the one hand, we ensure that we are not going too slowly to the optimal point on $f_1$, and on the other hand we do not want to overshoot too much on $f_2$.

Assuming Conjecture~\ref{conj:GM_f_TD} holds true, one can show that the optimal step size is an increasing function of $N$ with $3/2\leq h_{\text{opt}}(N)<2$ and $h_{\text{opt}}(N)\rightarrow 2$ as $N\rightarrow \infty$. More precisely, working out the expression defining $h_\text{opt}$ gives the following tight lower and upper estimates:
\begin{align}
2 - \frac{\log 4N}{2N} \sim 1+(1+4N)^{-1/(2N)} \leq h_\text{opt}(N) \leq 1+(1+2N)^{-1/(2N)} \sim 2 - \frac{\log 2N}{2N}.\label{eq:upper_low_GM_smooth}
\end{align}

It is interesting to compare the results {from the relaxation (G') proposed for GM in~\cite{Article:Drori}} with ours, for values of the normalized step size $h$ that are close to $h_{\text{opt}}$. Indeed, while the results of the two formulations are {quite similar for most values of $h$}, it turns out that those from~\cite{Article:Drori} are significantly more conservative in the zone around $h_{\text{opt}}$, as presented in Table~\ref{table:GM_f_WC} for different values of $N$. This also formally establishes the fact that the formulation from~\cite{Article:Drori} is a strict relaxation of the performance estimation problem.

   \renewcommand{\arraystretch}{1.1}
   \begin{table}[ht!]                                                                                          
\centering                                                                                                 
\begin{tabular}{c|c l| l c| l c}                                                        
                                                                                                                               
  N & $h_{\text{opt}}$ & Conjecture~\ref{conj:GM_f_TD} & Value computed in~\cite{Article:Drori} & Rel. error & Value from~\eqref{PEPS:Pep_sdp} & Rel. error \\
\hline                                                                                                     
  1 & 1.5000 & $LR^2/8.00$ & $LR^2/8.00$ & 0.00 & $LR^2/8.00$ & 7e-09 \\                                
                                                                                                    
  2 & 1.6058 & $LR^2/14.85$ & $LR^2/14.54$ & 2e-02 & $LR^2/14.85$ & 5e-09  \\                            
                                                                                                    
  5 & 1.7471 & $LR^2/36.94$ & $LR^2/32.57$ & 1e-01 & $LR^2/36.94$ & 1e-08  \\                          
                                                                                                    
  10 & 1.8341 & $LR^2/75.36$ & $LR^2/59.80$ & 3e-01 &  $LR^2/75.36$ & 3e-08  \\                       
                                                                                                    
  20 & 1.8971 & $LR^2/153.77$ & $LR^2/109.58$ & 4e-01 & $LR^2/153.77$ & 6e-08  \\                    
                                                                                                    
  30 & 1.9238 & $LR^2/232.85$ & $LR^2/156.23$ & 5e-01 & $LR^2/232.85$ & 7e-08 \\                   
                                                                                                   
  40 & 1.9388 & $LR^2/312.21$ & $LR^2/201.10$ & 6e-01 & $LR^2/312.21$ & 3e-08  \\                   
                                                                                                     
  50 & 1.9486 & $LR^2/391.72$ & $LR^2/244.70$ & 6e-01 & $LR^2/391.72$ & 1e-07  \\  
                 
   100 & 1.9705 & $LR^2/790.22$ & $LR^2/451.72$ & 7e-01 & $LR^2/790.22$ & 1e-07           \\                                                                                        
\end{tabular}                                                                                              
\caption{Gradient Method with $\mu=0$, worst-case computed with relaxation from~\cite{Article:Drori} and worst-case
obtained by exact formulation~\eqref{PEPS:Pep_sdp} for the criterion $f(x_N)-f^*$. 
Error is measured relatively to the conjectured result. Results obtained with MOSEK~\cite{Article:Mosek}.}                        
\label{table:GM_f_WC}                                                                             
\end{table}
  \renewcommand{\arraystretch}{1}

These numerical results have been obtained with MOSEK, a standard semidefinite optimization solver. Despite convexity of the formulation, it might happen that the solution returned by such as solver is inaccurate, and in particular (slightly) infeasible. In that case, the objective value of the approximate primal (resp.\@ dual) solution is no longer guaranteed to be a lower (resp.\@ upper) bound on the exact optimal value, hence potentially negating the advantage of an exact convex formulation. For this reason, all numerical results reported in this section have been double checked with an interval arithmetic-based semidefinite optimization solver~\cite{Techreport:vsdp} that returns an interval that is guaranteed to contain the optimal value. These guaranteed bounds are reported in Table~\ref{table:GM_validation_example_obj_mu0} for the case $h=1.5$, which compares them with Conjecture~\ref{conj:GM_f_TD}.
   \renewcommand{\arraystretch}{1.1}
  \begin{table}[ht!]                                                                                                      
\centering                                                                                                             
\begin{tabular}{c|ccccccc}                                                                                                
N & $1$ & $2$ & $5$ & $10$ & $15$ & $20$ & $30$ \\                                                                     
\hline                                                                                                                 
Relative error (upper limit) & 2e-09 & 7e-10 & 2e-09 & 1e-09 & 9e-10 & 1e-09 & 9e-10 \\                                 
Conjecture & 1.2e-01 & 7.1e-02 & 3.1e-02 & 1.6e-02 & 1.1e-02 & 8.2e-03 & 5.5e-03 \\                        
Relative error (lower limit) & 2e-09 & 3e-09 & 9e-09 & 9e-08 & 2e-07 & 3e-07 & 9e-07 \\                        
\end{tabular}                                                                                                          
\caption{Gradient method with relative step size $h=1.5$: numerical values from Conjecture~\ref{conj:GM_f_TD} and relative error for the upper and lower limits of the guaranteed interval obtained numerically with~VSDP \cite{Techreport:vsdp} and SeDuMi \cite{Article:Sedumi}.}
\label{table:GM_validation_example_obj_mu0}                                                                                
\end{table}
  \renewcommand{\arraystretch}{1}

We can observe that the use of a verified solver does not impact our conclusions about the validity of the conjecture. Moreover, this table is typical of what we observed for all conjectures in this section: all numerical results reported were validated\footnote{Except for tests {where validation encountered} numerical difficulties, i.e for which VSDP returned no valid interval, which occurred more and more frequently as the value of the worst-case bound became close{r} to zero.}, and  in what follows we will no longer mention this verification explicitly.

Finally, we compare results obtained with Conjecture~\ref{conj:GM_f_TD} with {classical analytical bounds from the literature} for the GM with unit normalized step size $h=1$ (which is usually recommended, and sometimes called optimal). {The best analytical bound we could find, e.g. in \cite{bertsekas2015convex}, states that}
{\begin{equation}
f(x_N)-f_*\leq \frac{LR^2}{2}\frac{1}{N+1}.
\label{eq:GM_smooth_std_f}
\end{equation}
{This analytical bound is asymptotically worse by a factor of $2$ than the bound predicted by Conjecture~\ref{conj:GM_f_TD} with $h=1$. Similarly, one can investigate the effect of choosing the optimal normalized step size $h_{\text{opt}}(N)$ instead of $h=1$: Conjecture~\ref{conj:GM_f_TD} then predicts another improvement by a factor of $2$. These observations follow from the} asymptotic (large $N$) behaviors of the different worst-case bounds on $f(x_N)-f_*$, which can easily be computed:
\begin{align*}
\text{Conjecture~\ref{conj:GM_f_TD} with $h=1$}{\longrightarrow} \frac{LR^2}{2} \frac{1}{2N+1}, \quad \text{Conjecture~\ref{conj:GM_f_TD}  with $h=h_\text{opt}(N)$ } \underset{N\to\infty}{\longrightarrow} \frac{LR^2}{2}\frac{1}{4N+1}.
\end{align*}}
\subsubsection{A generalized conjecture for strongly convex functions}\label{sss:sc}
In view of the encouraging results obtained for the GM in the smooth case, we now study the behavior of the GM on the class of strongly convex functions $\mathcal{F}_{\mu,L}{(\mathbb{R}^d)}$ using our formulation \eqref{PEPS:Pep_sdp} with the same performance criterion, objective function accuracy. It turns out that the solution for every problem consisted again in a one-dimensional worst-case function ($\text{rank}\ G = 1$) of the same piecewise quadratic type. We therefore introduce the following general definitions for functions $f_{1,\tau}$ and 
$f_{2}$:{
\begin{align*}
f_{1,\tau}(x)&=\left\{\begin{array}{ll}
\frac{\mu}{2}x^2 + a_\tau |x| +b_\tau \quad  & \text{if } |x| \geq \tau \\[0.1cm]
\frac{L}{2}x^2 & \text{else},
\end{array}
\right.\\[0.2cm]
f_{2}(x)&=\frac{L}{2}x^2,\notag
\end{align*}
where scalars $a_\tau = (L-\mu) \tau$ and $b_\tau = -\bigl(\frac{L-\mu}{2}\bigr) \tau^2$ are chosen to ensure continuity of $f_{1,\tau}$ and its first derivative, and $\tau$ is a parameter that controls the radius of the central quadratic piece (with the largest curvature).} Although the value of parameter $\tau$ could in principle be estimated from the numerical solutions of our problems, it turns out it can be computed {analytically} {by} maximizing the {final objective value} $f_{1,\tau}(x_N)$ {(assuming that all iterates stay in the affine zone $|x| \ge \tau$), which} then leads to  {
\begin{equation} \tau=\frac{R \kappa }{(\kappa-1)+(1-\kappa h)^{-2N}}\label{deftau}\end{equation}}
where $\kappa = \frac{\mu}{L}$ is the inverse condition number of the problem class $f\in\mathcal{F}_{\mu,L}{(\mathbb{R}^d)}$. We are now able to extend Conjecture~\ref{conj:GM_f_TD} to the GM applied to strongly convex functions.
{\begin{conjecture} {Any sequence of iterates $\left\{x_i\right\}$ generated by the gradient method GM with constant normalized step sizes $0\leq h \leq 2$ on a smooth strongly convex function} $f\in\mathcal{F}_{\mu,L}{(\mathbb{R}^d)}$ satisfies
\[ f(x_N)-f_*\leq {{\frac{LR^2}{2}}} \max\left(\frac{\kappa}{(\kappa-1) +(1-\kappa h)^{-2N}},(1-h)^{2N} \right). \] \label{conj:GM_f_mu}
\end{conjecture}}
{As in the previous section, this conjecture states} that the worst-case behavior of the GM according to objective function accuracy is achieved by function $f_{1,\tau}$ or $f_2$, depending on which of the two is worse. Proceeding now to its numerical validation, we first point out that our results are intrinsically limited to the accuracy that can be reached by numerical SDP solvers. For this reason, we only report on situations for which Conjecture~\ref{conj:GM_f_mu} predicts a final accuracy larger than $10^{-6}$,  ensuring a few significant digits for the numerical results. The resulting estimated relative differences between Conjecture~\ref{conj:GM_f_mu} and the numerical results obtained with~\eqref{PEPS:Pep_sdp} are given in Table~\ref{table:Gradient_Total_rel_f}, for different values of $\kappa$. We observe that the {c}onjecture is very well supported by our numerical results, with a largest relative error around $10^{-6}$, reached for the largest value of $\kappa$ considered here. This is expected as GM tends to perform better as $\kappa$ increases ({i.e.,}\@ final accuracy $f(x_N)-f_*$ approaches zero), which renders a precise comparison between numerical results and the conjecture more and more difficult.    \renewcommand{\arraystretch}{1.1}
  \begin{table}[ht!]                                                                                                                                                         
\centering                                                                                                                                                                
\begin{tabular}{c|cccccccc}                                                                                                                                      
                                                                                                                                                        
 $\kappa=$ & $0$ & $.001$ & $.005$ & $.010$ & $.015$ & $0.1$ & $0.2$ & $0.5$ \\
\hline                                                                                                                                                                                                                                                                                          
Rel. error & 6e-10 & 7e-10 & 4e-10 & 6e-10 & 8e-10 & 2e-07 & 9e-08 & 1e-06
\end{tabular}                                                                                                                                                             
\caption{Maximum relative estimated differences between Conjecture~\ref{conj:GM_f_mu} and corresponding numerical results obtained with SeDuMi~\cite{Article:Sedumi}. The maximum is taken over all $N\in\left\{1,\hdots,30\right\}$ and $h\in\left\{0.05,\hdots,1.95\right\}$ for which the conjecture predicts a worst-case larger than $10^{-6}$.}
\label{table:Gradient_Total_rel_f}                                                                                                                                              
\end{table}      
  \renewcommand{\arraystretch}{1}                                                      

We now investigate some consequences of {our} conjecture. First, we note that Conjecture~\ref{conj:GM_f_mu} tends to Conjecture~\ref{conj:GM_f_TD} as $\mu$ tends to zero. {This is a consequence of the fact that $\tau$ tends to $\frac{R}{2Nh+1}$ as $\kappa$ tends to zero (one can also check that function $f_{1,\tau}$ tends to function $f_1$ introduced earlier).} Hence our formulation~\eqref{PEPS:Pep_sdp} {closes} an apparent gap between worst-case analyses {of} the smooth convex and the smooth strongly convex cases. Indeed, to the best of our knowledge, existing worst-case bounds for the smooth strongly convex case do not converge to the smooth case as $\mu \to 0$. 

{It is also interesting to compare our results to those obtained with the IQC methodology of \cite{lessard2014analysis}. If we only care about asymptotic linear rates of convergence, Conjecture~\ref{conj:GM_f_mu} predicts \[   f(x_N)-f_* \leq  \frac{LR^2}{2}  \max \bigl\{ \kappa\, \rho_1^{2N} ,\rho_2^ {2N}  \bigr\}  \text{ with } \rho_1 = |1-\kappa h | \text{ and } \rho_2 = |1-h|  \]
(the first term in the $\max$ was obtained by neglecting  $(\kappa-1)$ in the denominator). On the other hand \cite[Section~4.4]{lessard2014analysis} proves that the distance to the solution converges 
linearly  according to  \[ \|x_N - x_*\| \le \rho^N \|x_0 - x_*\| \text{ with a factor } \rho = \max \{ \rho_1 ,\rho_2 \} \] with the same values for $\rho_1$ and $\rho_2$. This matches our asymptotic rate up to a multiplicative constant.}

{A}s for Conjecture~\ref{conj:GM_f_TD}, our new Conjecture~\ref{conj:GM_f_mu} suggests optimal step sizes $h_{\text{opt}}(N,\kappa)$, which can be obtained by solving the {equation} (for $0<\kappa<1$)
\begin{equation} \label{eq:eqhoptsc} \frac{\kappa}{(\kappa-1)+(1-\kappa h_{\text{opt}})^{-2N}}=(1-h_{\text{opt}})^{2N} \end{equation}
(note that one recovers the previous equation for $h_\text{opt}$ when $\mu$ tends to zero).
For a given $N$, as $\kappa$ increases from $0$ to $1$, those optimal step sizes decrease from $h_{\text{opt}}(N,0)$ (optimal step size in the smooth case) to $h_{\text{opt}}(N,1)=1$ (the latter being expected since it can only correspond to the case of function $f_2$ in the original~\eqref{Intro:PEP}, for which the GM with $h=1$ converges in one iteration). For a given $\kappa$, we find that $h_{\text{opt}}(N,\kappa)$ increases as $N$ increases, as in the smooth convex case, according to the following lower and upper estimates 
{\begin{align}
1+\left(\frac{\kappa-1}{\kappa}  + \frac1\kappa \Bigl( \frac{1+\kappa}{1-\kappa} \Bigr)^{2N}  \right)^{-\frac{1}{2N}}\leq & h_\text{opt}(N,\kappa) \label{eq:upper_low_GM_smoothstrcvx} \leq \min\left\{1+\left(  \frac{(\kappa-1)}{\kappa}+\frac1\kappa (1-\kappa )^{-2N}  \right)^{-\frac{1}{2N}},\frac{2}{1+\kappa}\right\}
\end{align}}which both tend to $\frac{2}{1+\kappa}$ as $N$ increases {(the first term appearing in the $\min$ of the upper bound tends to $2-\kappa$, which is always greater than $\frac{2}{1+\kappa}$)}. {This limiting normalized step size $\frac{2}{1+\kappa}$} corresponds to step size $\frac{2}{L+\mu}$ that is often recommended for the GM{, and sometimes called optimal}.   

We now illustrate the improvements provided by Conjecture~\ref{conj:GM_f_mu} with respect to the classical analytical worst-case bound found in {the literature. When using normalized step size $h=\frac{2}{1+\kappa}$, iterates from} GM applied to functions in $\mathcal{F}_{\mu,L}{(\mathbb{R}^d)}$ {are known to satisfy (see \cite{Book:Nesterov} for example)}
\begin{equation}
f(x_N)-f_*\leq \frac{LR^2}{2}\left({\frac{1-\kappa}{1+\kappa}}\right)^{2N}{.}
\label{eq:GM_std_f}
\end{equation} {On the other hand, as the number of steps $N$ tends to infinity, the true worst-case predicted by Conjecture~\ref{conj:GM_f_mu} for the same step size asymptotically tends to $\frac{L R^2}{2} \bigl(\frac{1-\kappa}{1+\kappa}\bigr)^{2N}$, which is exactly the same as \eqref{eq:GM_std_f}. 
Indeed, one can check that this rate is equal to the second term appearing in the $\max$ of Conjecture~\ref{conj:GM_f_mu}, while the first term tends to $\frac{L R^2}{2} \kappa \bigl(\frac{1-\kappa}{1+\kappa}\bigr)^{2N}$ which is always smaller.}

{One can however do better using the optimal step size $h_\text{opt}$. Since it is not closed-form, we use the following approximate expression obtained after solving a suitable approximation of equation \eqref{eq:eqhoptsc}
\[ \tilde{h}_{\text{opt}}(N)=\frac{1+\kappa^{\frac{1}{2N}}}{1+\kappa^{1+\frac{1}{2N}}} \] (note that $\tilde{h}_{\text{opt}}(N)$ tends to $\frac{2}{1+\kappa}$ as $N$ grows), and find that Conjecture~\ref{conj:GM_f_mu} predicts a worst-case bound tending to
\[ \frac{LR^2}{2} \kappa^{\frac{1}{1+\kappa}} \left(\frac{1-\kappa}{1+\kappa}\right)^{2N} \] which improves the asymptotic rate by a factor $\bigl(\frac{1}{\kappa}\bigr)^{\frac{1}{1+\kappa}}$ (which can be shown to lie between $\frac{3}{4}\frac{1}{\kappa}$ and $\frac{1}{\kappa}$).}

\subsubsection{A conjecture on the gradient norm}
We now consider a  different performance criterion, given by the norm of the gradient computed at the last iterate. Numerical experiments with our formulation suggest that results similar to those presented in the previous sections can be obtained both 
in the smooth convex and smooth strongly convex cases, based again on one-dimensional piecewise quadratic worst-case functions. Using the same definition for functions $f_{1,\tau}$ and $f_2$, and choosing {now the parameter $\tau$ according to}
{\begin{equation}\tau=\frac{R \kappa}{(\kappa-1)+(1-\kappa h)^{-N}}
,\label{deftau2}\end{equation}}we propose the following conjecture.
{\begin{conjecture} {Any sequence of iterates $\left\{x_i\right\}$ generated by the gradient method GM with constant normalized step sizes $0\leq h \leq 2$ on a smooth strongly convex function} $f\in\mathcal{F}_{\mu,L}{(\mathbb{R}^d)}$ satisfies
$$\norm{\nabla f(x_N)}_2\leq {L R}\max\left(\frac{\kappa}{(\kappa-1)+ (1 - \kappa h)^{-N} },|1-h|^{N}\right).$$
\label{conj:GM_g_mu}
\end{conjecture}}

As for Conjecture~\ref{conj:GM_f_mu}, we limit our numerical validation to the cases where the worst-case values predicted by the Conjecture are larger than $10^{-6}${; the largest relative error is about $10^{-7}$.}
 
We note that, as $\kappa$ tends to zero ({i.e.,}\@ the smooth case), Conjecture~\ref{conj:GM_g_mu} tends to  
$$\norm{\nabla f(x_N)}_2\leq {L R}\max\left(\frac{1}{Nh+1},|1-h|^{N} \right).$$
From that, we see that the optimal step size $h^\nabla_{\text{opt}}(N,0)$ for the GM is again an increasing function of $N$ with $\sqrt{2}\leq h^\nabla_{\text{opt}}(N,0) < 2$ and $h^\nabla_{\text{opt}}(N,0)\rightarrow 2$ as $N\rightarrow \infty$. {In the strongly convex case $\kappa > 0$, t}he optimal step size is a decreasing function of $\kappa$ and satisfies $h^\nabla_{\text{opt}}(N,\kappa)\rightarrow 1$ as $\kappa\rightarrow 1$. As in the previous case, $h^\nabla_{\text{opt}}(N,\kappa)$ is bounded above by $\frac{2}{1+\kappa}$, which we can confirm with the following lower and upper bounds on $h^\nabla_\text{opt}$:
\begin{align*}
1+\left(\frac{\kappa-1}{\kappa} + \Bigl( \frac{1+\kappa}{1-\kappa}\Bigr)^N\right)^{-1/N}\leq h^\nabla_{\text{opt}}(N,\kappa) \leq \min\left\{1+\left( \frac{\kappa-1}{\kappa} +\frac1\kappa (1-\kappa)^{-N} \right)^{{-1/N}},\frac{2}{1+\kappa}\right\}.
\end{align*}
In the smooth case, those bounds reduce to the simpler {expression}
\begin{align*}
2-\frac{\log{2N}}{N} \sim 1+\left(1+2N\right)^{-1/N} \leq h^\nabla_\text{opt} \leq 1+\left(1+N\right)^{-1/N} \sim 2 -\frac{\log{N}}{N}.
\end{align*}
We now compare with a standard analytical worst-case bound. {T}he iterates of the GM method {with normalized step size $\frac{2}{1+\kappa}$ are known to} satisfy \begin{equation}
 \norm{x_N-x_*}_2\leq R\left(\frac{1-\kappa}{1+\kappa}\right)^{{N}\!} {\text{ and }}\quad  \norm{\nabla f(x_N)}_2\leq L\norm{x_N-x_*}_2\leq LR\left(\frac{1-\kappa}{1+\kappa}\right)^{{N}\!} \label{eq:GM_std_g}
\end{equation}
{(see for example~\cite{Book:Nesterov} for the left inequality, and use the $L$-Lipschitz property of the gradient along with $\nabla f(x_*)=0$ to derive the right inequality). The latter estimate is tight according to Conjecture~\ref{conj:GM_g_mu}. Using the following approximate optimal step size}
\[ \tilde{h}_{\text{opt}}^\nabla(N)=\frac{1+\kappa^{\frac{1}{N}}}{1+\kappa^{1+\frac{1}{N}}} \] {(which tends to $\frac{2}{1+\kappa}$ as $N$ grows)
can be shown to improve the conjectured asymptotic rate by the same factor $\kappa^{-\frac{1}{1+\kappa}}$ as in the previous section.}

\subsection{Fast gradient method and optimized gradient method}
In this section we assess the performance in the smooth convex case {(}$\mu=0${)} of two accelerated first-order methods: the so-called fast gradient method (FGM) due to Nesterov~\cite{Nesterov:1983wy},  and an optimized gradient method (OGM) recently proposed by Kim and Fessler~\cite{kim2014optimized}.

\begin{center}
\fbox{
\parbox{0.98\textwidth}{

        \textbf{Fast Gradient Method (FGM)}
  \begin{itemize}
  \item[] Input: $f\in\mathcal{F}_{0,L}{(\mathbb{R}^d)}$, $x_0\in\mathbb{R}^d$, $y_0=x_0$, $\theta_0=1$.\\[-0.2cm]
  \item[] For $i=0:N-1$\\[-0.5cm]
      \begin{align*}
    &y_{i+1}=x_i-\frac{1}{L}\nabla f(x_i)\\
    &\theta_{i+1}= \frac{1+\sqrt{4\theta_i^2 +1}}{2}\\
    &x_{i+1}=y_{i+1}+\frac{\theta_i-1}{\theta_{i+1}} (y_{i+1}-y_{i})
    \end{align*}
  \end{itemize}
}}
\end{center}
\begin{center}
\fbox{
\parbox{0.98\textwidth}{
        \textbf{Optimized Gradient Method (OGM)}
  \begin{itemize}
  \item[] Input: $f\in\mathcal{F}_{0,L}{(\mathbb{R}^d)}$, $x_0\in\mathbb{R}^d$, $y_0=x_0$, $\theta_0=1$.\\[-0.2cm]
  \item[] For $i=0:N-1$\\[-0.5cm]
      \begin{align*}
    y_{i+1}&=x_i-\frac{1}{L}\nabla f(x_i)\\
    \theta_{i+1}&=\left\{
    \begin{array}{ll}
    \frac{1+\sqrt{4\theta_i^2 +1}}{2}, & i\leq N-2\\
    \frac{1+\sqrt{8\theta_i^2 +1}}{2}, & i=N-1 
    \end{array}
    \right.\\
    x_{i+1}&=y_{i+1}+\frac{\theta_i-1}{\theta_{i+1}} (y_{i+1}-y_{i}) +\frac{\theta_i}{\theta_{i+1}} (y_{i+1}-x_i)
    \end{align*}
  \end{itemize}
}}
\end{center}
Both of these algorithms are defined in terms of  two sequences: $\left\{y_i\right\}_i$ is a primary sequence, and $\left\{x_i\right\}_i$ is a secondary sequence, where the gradient is evaluated. We first show that both of these algorithms can be expressed as fixed-step first-order methods{, which we defined as}
$${x_{i}=x_{0}-\sum_{k=0}^{i-1} h_{i,k} \nabla f(x_{k}) \quad \text{(for $L=1$)}.}$$
{One way to proceed is to} focus on the secondary sequence  $\left\{x_i\right\}_i$ and substitute the $y_i$'s in the algorithm formulation. For FGM{, we have}
\begin{align*}
x_{i+1}&=x_i-\frac{g_i}{L}+\frac{\theta_i -1}{\theta_{i+1}}\left( x_i-x_{i-1}-\frac{g_i}{L}+\frac{g_{i-1}}{L}\right),\\
&=x_i+\frac{\theta_i-1}{\theta_{i+1}} (x_i-x_{i-1}) -\left(\frac{\theta_i-1}{\theta_{i+1}}+1\right)\frac{g_i}{L}+\frac{\theta_i-1}{\theta_{i+1}}\frac{g_{i-1}}{L},
\end{align*}
which allows {to obtain} the step sizes relative to $x_0$ by recurrence:
\begin{align*}
h_{i+1,k}=\left\{\begin{array}{ll}
h_{i,k}+\frac{\theta_{i}-1}{\theta_{i+1}}\left(h_{i,k}-h_{i-1,k}\right) \quad & \text{     if } k \leq i-2, \\
h_{i,k}+\frac{\theta_{i}-1}{\theta_{i+1}}(h_{i,k}-1) & \text{   if } k = i-1, \\
\frac{\theta_{i}-1}{\theta_{i+1}}+1 & \text{   if } k = i, \\
\end{array}\right.
\end{align*}
with initial conditions $h_{1,0}=1$, $h_{1,k}=0$ if $k<0$ and $h_{0,k}=0$ for all $k$. Similarly, {we have} for OGM
\begin{align*}
h_{i+1,k}=\left\{\begin{array}{ll}
h_{i,k}+\frac{\theta_i-1}{\theta_{i+1}}\left(h_{i,k}-h_{i-1,k}\right) \quad & \text{     if } k \leq i-2, \\
h_{i,k}+\frac{\theta_i-1}{\theta_{i+1}}(h_{i,k}-1) & \text{   if } k = i-1, \\
\frac{2\theta_i-1}{\theta_{i+1}}+1 & \text{   if } k = i, \\
\end{array}\right.
\end{align*}
with the same initial conditions. This approach {will provide} estimates for the last secondary iterate $x_N$. If an estimate for last primary iterate $y_N$ is needed, one {just has} to replace the expression of $x_N$ by $y_N$, which is done by using the following alternative coefficients for the last step:
\begin{align*}
h_{N,k}=\left\{\begin{array}{ll}
h_{N-1,k} & \text{   if } k \leq N-2, \\
1 & \text{   if } k=N-1, \\
\end{array}\right.
\end{align*}
for both FGM and OGM.

Again, our numerical experiments strongly suggest the same assumption about the shape of the worst-case functions, {i.e.,}\@ one-dimensional and piecewise quadratic ({with} iterates staying in the affine zone of $f_{1,\tau}$). Using this {property}, we are able {to} compute {the following} values of $\tau$ achieving the worst-case final {objective} accuracy{, which surprisingly hold for both the classical FGM and the more recent OGM (a coincidence for which we can offer no explanation)}
{\[ \tau_1=\frac{R}{2\sum_{k=0}^{N-2}h_{N-1,k} +3} \text{ for the primary sequence, } \tau_2=\frac{R}{2\sum_{k=0}^{N-1}h_{N,k} +1}\text{ for the secondary sequence.}\]}Our numerical results suggest the following {two}  conjectures ({validations for both conjectures were performed for values of} $N\in \left\{1,\hdots,100\right\}$ and {displayed} a relative error {less than} $10^{-4}$). 
{\begin{conjecture} 
Any (primary) sequence of iterates $\left\{y_i\right\}$ generated by the fast gradient methd FGM (resp. optimized gradient method OGM) on a smooth convex function $f\in\mathcal{F}_{0,L}{(\mathbb{R}^d)}$ satisfies
\[ f(y_N)-f_*\leq f_{1,\tau_1}(y_{1,N})=\frac{LR^2}{2}\frac{1}{2\sum_{k=0}^{N-2} h_{N-1,k}+3}, \]
where $y_{1,N}$ is the final (primary) iterate computed by FGM (resp. OGM) applied to $f_{1,\tau_1}$ starting from $x_0=R$, and quantities $h_{N-1,k}$ are the fixed coefficients of the last step of FGM (resp. OGM).
\label{conj:FGM_OGM_prim}
\end{conjecture}}
{\begin{conjecture} 
Any (secondary)  sequence of iterates $\left\{x_i\right\}$ generated by the fast gradient methd FGM (resp. optimized gradient method OGM) on a smooth convex function $f\in\mathcal{F}_{0,L}{(\mathbb{R}^d)}$ satisfies
\[ f(x_N)-f_*\leq f_{1,\tau_2}(x_{1,N})=\frac{LR^2}{2}\frac{1}{2\sum_{k=0}^{N-1} h_{N,k}+1}, \]
where $x_{1,N}$ is the final (secondary) iterate computed by FGM (resp. OGM) applied to $f_{1,\tau_2}$ starting from $x_0=R$, and quantities $h_{N,k}$ are the fixed coefficients of the last step of FGM (resp. OGM).
\label{conj:FGM_OGM_sec}
\end{conjecture}}

{The worst-case bounds in these two conjectures involve the normalized step sizes of the FGM and OGM methods. It turns out these can be computed in closed form for OGM (see also \cite{kim2015optimized}), and give ($N\geq 1$)
\[ f(y_N)-f_*\leq \frac{LR^2}{4\theta_{N-1}^2+2}\leq \frac{LR^2}{2} \frac{2}{(N+1)^2+2}\text{  and  } f(x_N)-f_*\leq \frac{LR^2}{2\theta_N^2} \leq \frac{LR^2}{2} \frac{2}{(N+1)(N+1+\sqrt{2})}  \]
(where the inequalities rely on the bounds $\theta_{N-1}^2 \geq \frac{(N+1)^2}{4}$ and $\theta_N^2 \ge \frac{(N+1)(N+1+\sqrt{2})}{2}$). We were not able to obtain similar closed-form bounds for the FGM.}

{We now} compare the numerical values obtained with Conjectures~\ref{conj:FGM_OGM_prim} and~\ref{conj:FGM_OGM_sec} with analytical bounds {known} for the FGM. 
{We use for the primary sequence }
\begin{equation}
f(y_N)-f_*\leq \frac{2LR^2}{(N+1)^2},
\label{eq:FGM_Beck}
\end{equation}
which can be found in{~\cite{beck2009fast}, and for the secondary sequence}
\begin{equation}
f(x_N)-f_*\leq \frac{2LR^2}{(N+2)^2}
\label{eq:FGM_Kim}
\end{equation}
which was very recently derived in~\cite{kim2014optimized}.
The comparison is {displayed on} \figref{Fig:StdBdGd_vs_conj45}. The asymptotic behaviors of {both} sequences are well captured by the analytical bounds~(\ref{eq:FGM_Beck}) and~(\ref{eq:FGM_Kim}), but we observe that the estimation of the transient worst cases are improved by our conjectures: a factor approximately {equal to} $1.15$ is gained for both sequences after $30$ iterations.

{Before going into the next section, we comment on the applicability of our results to monotone variants of first-order methods, i.e. methods which guarantee $f(y_{i+1})\leq f(y_i)$. Consider for example FISTA~\cite{beck2009fast}, which is equivalent to FGM when applied to smooth unconstrained minimization.  MFISTA~\cite{beck2009fastmonotone}, a monotone variant of FISTA, happens to generate a monotonically decreasing sequence $\left\{f(y_i)\right\}_i$ when applied to our worst-case function $f_{1,{\tau_1}}$ from $x_0=R$. This means that the corresponding lower bound from Conjecture 4 also applies to MFISTA.}

\begin{figure}[!ht]
\begin{center}
\includegraphics[scale=0.5]{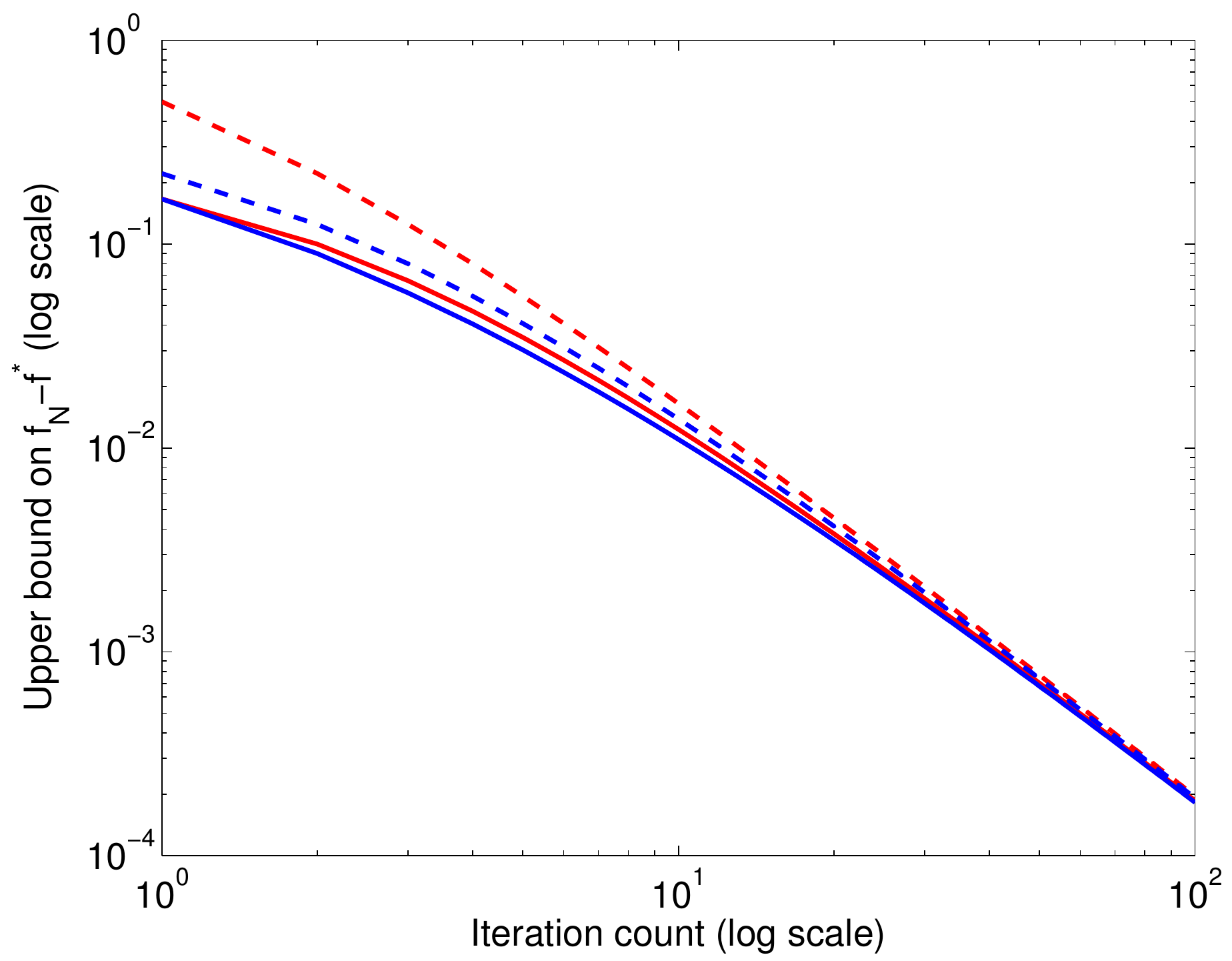}
\caption{{Comparison of the worst-case performance of the FGM: analytical bound~(\ref{eq:FGM_Beck}) (dashed red) versus Conjecture~\ref{conj:FGM_OGM_prim} (red) and analytical bound~(\ref{eq:FGM_Kim}) (dashed blue) versus Conjecture~\ref{conj:FGM_OGM_sec} (blue)}.}
\label{Fig:StdBdGd_vs_conj45}
\end{center}
\end{figure}

\subsection{{Estimation of the smallest gradient norm among all iterates}}
\label{subsec:numerics_mingrad}
First-order methods are often used in dual approaches where, in addition to objective function accuracy, gradient norm plays an important role. Indeed, this quantity controls primal feasibility of the iterates (see {e.g.,}\@ \cite{devolder2012double}). 
Considering for example the accelerated FGM in the smooth case, we know from the previous section
 that the {classical analytical bound on the} worst-case accuracy for a function in $\mathcal{F}_{0,L}{(\mathbb{R}^d)}$ is {given by} $\frac{2 L R^2}{(N+1)^2}$. From that bound, it is easy to obtain a similar bound on the last gradient {norm}, using Corollary~\ref{corr:Linterp}:
\begin{equation}
\norm{\nabla f(y_N)}_2\leq \sqrt{2L (f(y_N)-f_*)}\leq \frac{2LR}{N+1}.
\label{eq:naive_FGM}
\end{equation}
Observe that this asymptotic rate is significantly worse than that of the objective function accuracy{, and not better than that of the gradient method GM (see Conjecture~\ref{conj:GM_g_mu})}.

However, {it is well-known} that {the norm of the gradient} is not  decreasing monotonically among iterates of the FGM. Hence, in this section, we will estimate the worst-case performance of FGM according to the {smallest} observed gradient norm among all iterates:
$$\min_{i\in\left\{0,\hdots,N\right\}} \norm{\nabla f(y_i)}_2.$$
In order to do so, only a slightly modified version of~\eqref{PEPS:Pep_sdp} is needed: this min-type objective function is representable using a new variable $t$ for the objective and $N+1$ additional linear inequalities $t \le \norm{\nabla f(y_i)}^2_2 \Leftrightarrow t \le G_{i,i}$ {for all $0 \le i \le N$}. Note that {the maximum is still attained} since this concave piecewise linear objective function {is} continuous. 

This criterion was suggested in~\cite{nesterov2012make}, which proposes a variant of FGM that consists in performing $N/2$ steps of the standard FGM followed by $N/2$ steps of the GM with $h=1$. It is then theoretically established that this variant of FGM{, which we denote by MFGM,} satisfies
\begin{equation}
\min_{i\in\left\{0,\hdots,N\right\}} \norm{\nabla f(y_i)}_2 \leq \frac{8LR}{N^{3/2}}{,}
\label{eq:bestgrad_FGM}
\end{equation}
an improvement compared to the rate of convergence of the {gradient of the last iterate}. 

We now compare FGM with this modified variant MFGM using our performance estimation formulation.
\figref{Fig:mingrad_FGM} compares the behaviors of those methods in both their last (for FGM) and  best iterates, as well as the above analytic bounds \eqref{eq:naive_FGM} and \eqref{eq:bestgrad_FGM}.
\begin{figure}[!ht]
\begin{center}
\includegraphics[scale=0.5]{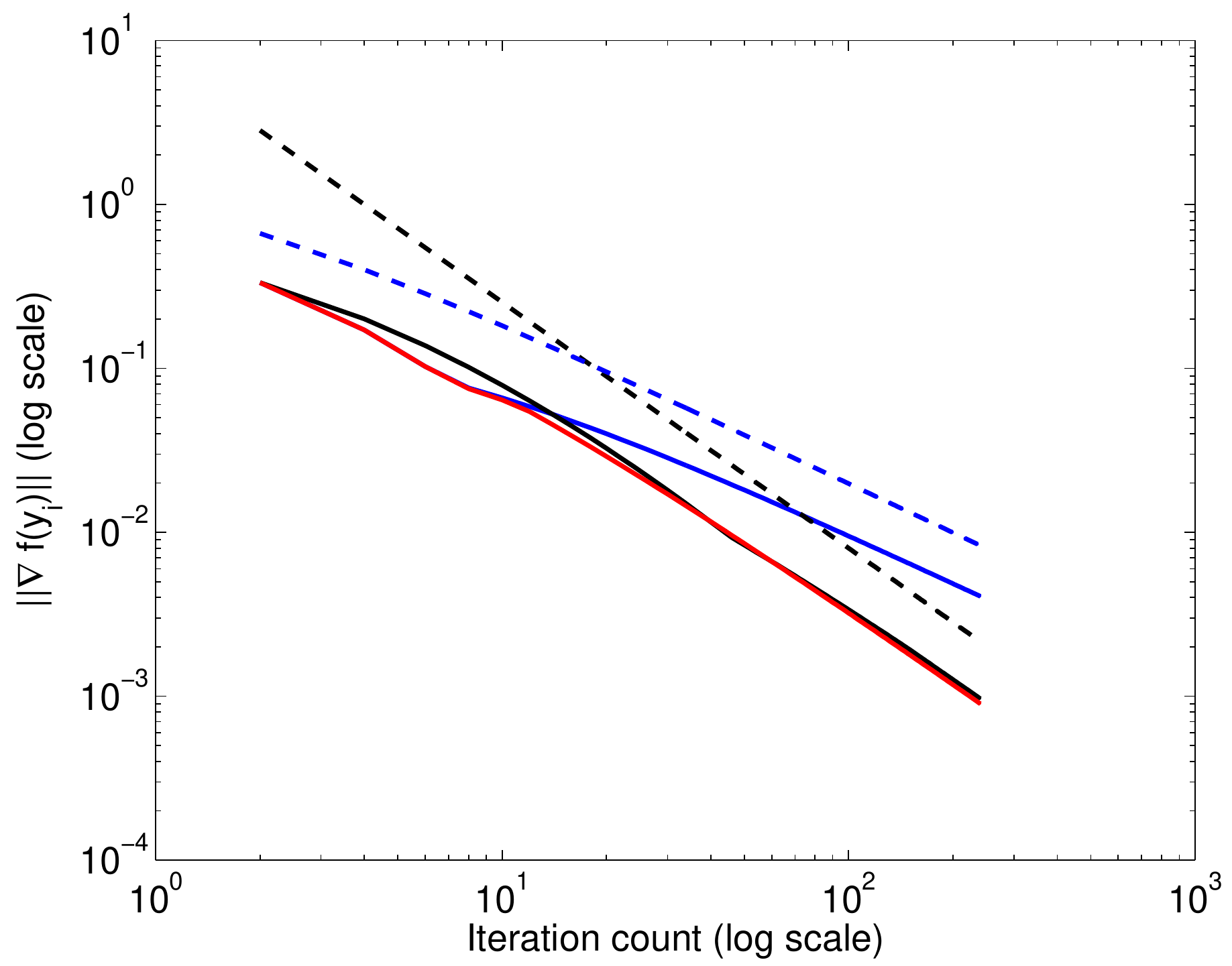}
\caption{Comparison of gradient norm convergence rates for the FGM and the MFGM from~\cite{nesterov2012make}. Theoretical guarantees are dashed. {Analytical bound on FGM~(\ref{eq:naive_FGM}) in its last iterate (dashed blue); numerical worst-case for FGM at its last iterate (blue); numerical worst-case for FGM at its best iterate (red); analytical bound on MFGM~(\ref{eq:bestgrad_FGM}) for the best iterate (dashed black); numerical worst-case for MFGM at its best iterate (black)}.}
\label{Fig:mingrad_FGM}
\end{center}
\end{figure}
This experiment confirms that the gradient norm of the last iterate of FGM decreases according to the {slower $\mathcal{O}(N^{-1})$ rate of \eqref{eq:naive_FGM}}. We also observe that {both the MFGM and the original FGM achieve the same $\mathcal{O}(N^{-3/2})$ convergence rate for the smallest gradient norm, which was not known before for FGM}. In addition, numerical results {reported} in Table~\ref{table:FGMandMFGM_grad} suggest that FGM performs slightly better than MFGM.

   \renewcommand{\arraystretch}{1.2}
  \begin{table}[ht!]                                                                                                                                     
\centering  
\begin{tabular}{c|lllll}                                                                                                                      
 N & FGM, analytic (\ref{eq:naive_FGM}) & FGM, last & FGM, best & MFGM, analytic (\ref{eq:bestgrad_FGM}) & MFGM, best \\     
\hline2 & $LR/1.50$ & $LR/3.00$ & $LR/3.00$ & $LR/0.35$ & $LR/3.00$ \\                                                                                                                                              
  4 & $LR/2.50$ & $LR/5.84$ & $LR/5.84$ & $LR/1.00$ & $LR/5.00$ \\                                                                                  
  10 & $LR/5.50$ & $LR/15.14$ & $LR/15.62$ & $LR/3.95$ & $LR/12.66$ \\                                                                                
  20 & $LR/10.50$ & $LR/25.08$ & $LR/34.49$ & $LR/11.18$ & $LR/30.77$ \\                                                                              
  30 & $LR/15.50$ & $LR/35.13$ & $LR/58.50$ & $LR/20.54$ & $LR/55.38$ \\                                                                             
  40 & $LR/20.50$ & $LR/45.19$ & $LR/86.17$ & $LR/31.62$ & $LR/86.41$ \\                                                                             
  50 & $LR/25.50$ & $LR/55.25$ & $LR/117.08$ & $LR/44.19$ & $LR/119.63$ \\                                                                           
  100 & $LR/50.50$ & $LR/105.49$ & $LR/311.34$ & $LR/125.00$ & $LR/296.58$ \\                                                                          
  200 & $LR/100.50$ & $LR/205.77$ & $LR/850.59$ & $LR/353.55$ & $LR/791.87$ \\                                                                                                                                                                                                                  
\end{tabular}                                                                                                                                         
\caption{FGM and MFGM: comparison between theoretical bounds and numerical results for the criteria $\norm{\nabla f({\color{red} \mathbf{y_N}})}_2$(last) and $\min_i\norm{\nabla f({\color{red} \mathbf{y_i}})}_2$ (best). Results obtained with~\cite{Article:Mosek}.}
\label{table:FGMandMFGM_grad}                                                                                                                                     
\end{table}   
  \renewcommand{\arraystretch}{1}
\newpage

A regularization technique is also describe{d} in~\cite{nesterov2012make}, {featuring} a $\mathcal{O}(N^{-2})$ convergence rate up to a logarithmic factor. A drawback of this approach is that it requires a bound on the distance to the optimal solution, and that the coefficients of the method explicitly depend on this bound. No fixed-step method achieving the same $\mathcal{O}(N^{-2})$ seems to be known.

\section{Conclusion}
\label{sec:ccl}

The contribution of this paper is threefold: first, we present necessary and sufficient conditions for smooth strongly convex interpolation. Those conditions are derived by showing an explicit way of constructing the interpolating functions.
Second, we show that the exact worst-case performance of any fixed-step first-order algorithm for smooth strongly convex {unconstrained} optimization {can} be formulated as a convex problem. In this context, {our interpolation procedure also provides explicit functions achieving the worst-case bounds computed by our approach}. Third, we {test of our formulation numerically} on a variety of functions classes, {first-order} methods and performance criteria, establishing on the way a series of conjectures on the corresponding worst-case behaviors. In particular, we suggest new tight estimates of the optimal step size for the fixed-step gradient method {with constant step size}, which depend on the number of iterations and the condition number.

Our {p}erformance {e}stimation problem provide a generic tool to analyze fixed-step first-order methods. It allows computing both exact worst-case guarantees and functions reaching them, and provides a unified algorithmic analysis for smooth convex functions and smooth strongly convex functions.

The exact worst-case values provided by our approach require solving a convex {semidefinite} program  whose size grows as the square of the number of iterations considered, which may become prohibitive when this number of iterations is large. 
This can be avoided using iteration-independent bounds, as proposed in~\cite{lessard2014analysis}, but at the cost of obtaining poorer worst-case guarantees. 
  
Further improvements to our approach  include an extension of~\eqref{Intro:PEP} to more general methods, such as first-order methods equipped with line search, or first-order methods designed to work on a restricted convex feasible region (projected gradient). Another desirable feature would be the ability to optimize the step sizes of the method considered in \eqref{PEPS:Pep_sdp}, as was proposed in~\cite{Article:Drori,kim2014optimized} for the relaxed version of~\eqref{Intro:PEP}.

\textbf{Software.} Our semidefinite programming approach to performance estimation has been implemented with MATLAB code, which can be downloaded from \verb?http://perso.uclouvain.be/adrien.taylor?. 
This routine features an easy-to-use interface, which allows the estimation of the worst-case performance of a given fixed-step first-order algorithm (to be chosen among a pre-defined list or to be specified by its coefficients) on a given class of functions, for a given performance criterion, after any number of steps.

\bibliographystyle{spmpsci}      
\bibliography{bib_}{}   

\end{document}